\documentclass[12pt]{article}
\usepackage{graphicx}
\usepackage{amsmath,amsfonts,amssymb,amsthm,bbm}
\title{A Brunn--Minkowski theory\\ for coconvex sets of finite volume}
\author{Rolf Schneider}
\date{}
\newcommand{\Sn}{{\mathbb S}^{n-1}}

\newcommand{\R}{{\mathbb R}}

\newcommand{\C}{{\mathcal C}}

\newcommand{\K}{{\mathcal K}}

\newcommand{\Rn}{{\mathbb R}^n}

\newcommand{\N}{{\mathbb N}}

\newcommand{\Ha}{\mathcal{H}}

\newcommand{\B}{\mathcal{B}}
\newcommand{\D}{{\rm d}}

\newcommand{\A}{{\mathcal A}}

\newcommand{\cL}{{\mathcal L}}

  \renewcommand{\exp}{{\rm exp}\,}

\newtheorem{theorem}{Theorem}
\newtheorem{lemma}{Lemma}
\newtheorem{definition}{Definition}

\begin{document}
\maketitle

\begin{abstract}
Let $C$ be a closed convex cone in $\R^n$, pointed and with interior points. We consider sets of the form $A=C\setminus K$, where $K\subset C$ is a closed convex set. If $A$ has finite volume (Lebesgue measure), then $A$ is called a $C$-coconvex set, and $K$ is called $C$-close. The family of $C$-coconvex sets is closed under the addition $\oplus$ defined by $C\setminus(A_1\oplus A_2)= (C\setminus A_1)+(C\setminus A_2)$. We develop first steps of a Brunn--Minkowski theory for $C$-coconvex sets, which relates this addition to the notion of volume. In particular, we establish the equality condition for a Brunn--Minkowski type inequality (with reversed inequality sign) and introduce mixed volumes and their integral representations. For $C$-close sets, surface area measures and cone-volume measures can be defined, in analogy to similar notions for convex bodies. They are Borel measures on the intersection of the unit sphere with the interior of the polar cone of $C$. We prove a Minkowski-type uniqueness theorem for $C$-close sets with equal surface area measures. Concerning Minkowski-type existence problems, we give conditions for a Borel measure to be either the surface area measure or the cone-volume measure of a $C$-close set. These conditions are sufficient in the first case, and necessary and sufficient in the second case.\\[1mm]
\end{abstract}

\section{Introduction}\label{sec1}

Let $C$ be a pointed closed convex cone with apex $o$ and with interior points in Euclidean space $\R^n$. This cone will be fixed throughout the following. A closed convex set $K\subset C$ is called $C$-{\em close} if $C\setminus K$ has positive finite Lebesgue measure, and in this case we say that $C\setminus K$ is a $C$-{\em coconvex} set. It should be observed that a $C$-close set can be entirely contained in the interior of $C$. 

Let $A_0=C\setminus K_0$, $A_1=C\setminus K_1$ be $C$-coconvex sets. Their {\em co-sum} is defined by 
$$ A_0\oplus A_1 = C\setminus(K_0+K_1),$$
where $+$ denotes the usual Minkowski addition. Note that $ K_0+K_1\subset C+C=C$. Whereas the Minkowski sum of two unbounded closed convex sets need not be closed in general, it is easy to see that $K_0+K_1$ is closed, because $K_0,K_1$ are subsets of a pointed cone. That $A_0\oplus A_1$ has finite volume, is a consequence of the following `complemented Brunn--Minkowski inequality'. Here, $\lambda A:= \{\lambda a:a\in A\}$ for $\lambda\ge 0$ and a $C$-coconvex set $A$. By $V_n$ we denote the volume (Lebesgue measure).

\begin{theorem}\label{T1}
Let  $A_0,A_1$ be $C$-coconvex sets, and let $\lambda\in (0,1)$. Then
\begin{equation}\label{1.1}
V_n((1-\lambda)A_0\oplus\lambda A_1)^{\frac{1}{n}} \le (1-\lambda)V_n(A_0)^{\frac{1}{n}}+\lambda V_n(A_1)^{\frac{1}{n}}.
\end{equation}
Equality holds if and only if $A_0=\alpha A_1$ with some $\alpha>0$.
\end{theorem}

The impetus for this investigation came from two papers by Khovanski\u{\i} and Timorin \cite{KT14} and by Milman and Rotem \cite{MR14}, who studied different aspects of complemented versions of classical inequalities (that is, versions for (relative) complements of convex (or more general) sets). Let $\Delta\subset C$ be a closed convex set such that $C\setminus \Delta$ is {\bf bounded} and nonempty. Khovanski\u{\i} and Timorin \cite{KT14} call the set $C\setminus(\Delta\cup\{o\})$ a {\em coconvex body}. (The non-inclusion of certain boundary points is relevant for some of their aims, but not if volumes are considered.) To indicate their motivation, the authors `briefly overview the connections of convex geometry with algebraic geometry, of algebraic geometry with singularity theory and, finally, of singularity theory with coconvex geometry' (citation). Then they go on to obtain complemented versions of the main inequalities of the classical Brunn--Minkowski theory. These include the Aleksandrov--Fenchel inequalities, the Brunn--Minkowski inequality, Minkowski's first and second inequality. The derivation of the Aleksandrov--Fenchel inequalities for coconvex bodies from those for convex bodies is brief and particularly elegant. 

When volumes of sets $C\setminus K$ are studied, it seems natural to admit convex sets $K$ for which $C\setminus K$ has finite volume, without necessarily being bounded. This is what we do here. Of course, the inequality (\ref{1.1}) could be obtained by approximation from the results in \cite{KT14}, but we do not see a possibility to get the equality condition in this way. This equality condition is crucial for a subsequent application. Our proof of (\ref{1.1}), which yields the equality condition, adapts the classical Kneser--S\"uss approach to the Brunn--Minkowski inequality for convex bodies, but needs extra steps, since we deal also with unbounded sets.
 
The second incentive, the work of Milman and Rotem \cite{MR14}, which was inspired by Borell's theory of convex measures, established a complemented Brunn--Minkowski inequality for complements of general sets, with Lebesgue measure replaced by more general measures, and deduced isoperimetric type inequalities. Again, inequality (\ref{1.1}) is a very special case, but it is not clear how to obtain equality conditions from their approach.

Our first application of the equality condition in Theorem 1 is a Minkowski type uniqueness theorem. In the development of the classical Brunn--Minkowski theory for convex bodies, some of the first steps are the introduction of mixed volumes, their integral representation, and consequences of the Brunn--Minkowski theorem, such as Minkowski's first and second inequality for mixed volumes. A first application then is the uniqueness result in the Minkowski problem concerning convex bodies with given surface area measures. We follow a similar line for $C$-coconvex sets or, what is equivalent but perhaps more convenient, for $C$-close sets. In particular, we prove a counterpart to Minkowski's uniqueness theorem. Let $K$ be a $C$-close set. Its area measure is defined as follows. Let $C^\circ$
be the polar cone of $C$. Denoting by $\Sn$ the unit sphere of $\R^n$, we define
$$ \Omega_C:= \Sn\cap {\rm int}\,C^\circ.$$
The spherical image $\sigma(K,\beta)$ of the closed convex set $K$ at the set $\beta$ is the set of all outer unit normal vectors of $K$ at points of $K\cap \beta$. For the $C$-close set $K$, we have $\sigma(K,{\rm int}\,C)\subseteq\Omega_C$, since a supporting hyperplane of $K$ at a point of ${\rm int}\,C\cap {\rm bd}\,K$ (where ${\rm bd}$ denotes the boundary) separates $K$ and the origin $o$. For $\omega\subseteq\Omega_C$, the reverse spherical image $\tau(K,\omega)$ is defined as the set of all points in ${\rm bd}\,K$ at which there exists an outer unit normal vector belonging to $\omega$. For Borel sets $\omega\subseteq\Omega_C$ one then defines
$$ S_{n-1}(K,\omega)= \Ha^{n-1}(\tau(K,\omega)),$$
where $\Ha^{n-1}$ is the $(n-1)$-dimensional Hausdorff measure (so that $S_{n-1}(K,\cdot)$ is the usual surface area measure, extended to closed convex sets). Using the theory of surface area measures of convex bodies (see \cite[Sect. 4.2]{Sch14}), it is easily seen that this defines a Borel measure on $\Omega_C$, the {\em surface area measure} $ S_{n-1}(K,\cdot)$ of $K$. In contrast to the case of convex bodies, the surface area measure of a $C$-close set is only defined on the open subset $\Omega_C$ of $\Sn$, and the total measure may be infinite. 

Now we can state a counterpart to Minkowski's uniqueness theorem. 

\begin{theorem}\label{T2}
If $K_0, K_1$ are $C$-close sets with $S_{n-1}(K_0,\cdot)= S_{n-1}(K_1,\cdot)$, then $K_0=K_1$.
\end{theorem}

The fact that $C\setminus K_i$ ($i=0,1$) has finite volume, is crucial for the proof. We do, however, not know whether it is essential for the theorem. In other words, does the uniqueness still hold if the condition that $C\setminus K_i$ has finite volume is replaced by the condition that $K_i$ is only `asymptotic' to $C$, in the sense that the distance of the boundaries of $C$ and $K_i$ outside $B(o,r)$ (ball with center $o$ and radius $r$) tends to zero, as $r\to\infty$?

In the theory of convex bodies, Minkowski's existence theorem is one of the fundamental results and is still finding constant interest. For Minkowski's classical theorem and its extension by Fenchel, Jessen and Aleksandrov, we refer to \cite[Sect. 8.2]{Sch14}, where the Notes describe later developments. Information on recent variants, such as those in the $L_p$ Brunn--Minkowski theory, can be found in Section 9.2 of \cite{Sch14} and its Notes. A counterpart to Minkowski's existence problem for coconvex sets would certainly be of interest. We formulate this question as follows. What are the necessary and sufficient conditions on a Borel measure $\mu$ on $\Omega_C$ such that there exists a $C$-close set $K$ with $S_{n-1}(K,\cdot)=\mu$? We can only give a special sufficient condition, leading to coconvex sets as considered by Khovanski\u{\i} and Timorin. We say that the closed convex set $K\subset C$ is $C$-{\em full} if $C\setminus K$ is bounded. 

\begin{theorem}\label{T3}
Every nonzero, finite Borel measure on $\Omega_C$, which is concentrated on a compact subset of $\Omega_C$, is the surface area measure of a $C$-full set.
\end{theorem}

Thus, no further conditions are required. The uniqueness is covered by Theorem \ref{T2}.

Conversely, however, it should be observed that there do exist $C$-full sets for which the surface area measure is not concentrated on a compact subset of $\Omega_C$. Moreover, there are $C$-close sets $K$ for which the surface area measure is finite, but $C\setminus K$ is not bounded. This points to some of the difficulties that are inherent to the general Minkowski existence problem for $C$-close sets. 

A completely satisfactory existence theorem can be proved if the surface area measure is replaced by the cone-volume measure. Recall that in the logarithmic Minkowski problem (see \cite{BLYZ13}), which is the case $p=0$ of the $L_p$ Minkowski problem, the role of the surface area measure is taken over by the cone-volume measure. This can also be defined for a $C$-close set $K$, as follows. For each point $x\in{\rm bd}\,K\cap  {\rm int}\,C$, the half-open line-segment $[o,x)$ with endpoints $o$ (the origin) and $x$ (excluded) is contained in $C\setminus K$ (see Sect. \ref{sec2} below). For a Borel set $\omega\subseteq\Omega_C$, let $V_K(\omega)$ denote the Lebesgue measure of the set
$$ \bigcup_{x\in\tau(K,\omega)} [o,x).$$
As shown in Section \ref{sec10}, this can be represented by
$$ V_K(\omega) = \frac{1}{n} \int_\omega -h_K(u)\,S_{n-1}(K,\D u),$$
where $h_K$ denotes the support function of $K$. $V_K$ is a measure on $\Omega_C$ and is called the {\em cone-volume measure} of $K$.

Cone-volume measures of convex bodies have been studied thoroughly during the last years; we mention here only \cite{Sta02, HSW05, Nao07, Sta08, Xio10, PW12, BLYZ13, HL14, Zhu14, BH15, BHZ16, BHk16, BHk17}. It is known that they have to satisfy some highly non-trivial properties. In view of this, Theorem \ref{T5} below may seem rather surprising, at first sight. First we show: 

\begin{theorem}\label{T4}
Every nonzero, finite Borel measure on $\Omega_C$, which is concentrated on a compact subset of $\Omega_C$, is the cone-volume measure of a $C$-full set.
\end{theorem}

From this result, we can deduce the following theorem, which now deals with $C$-close sets.

\begin{theorem}\label{T5}
Every nonzero, finite Borel measure on $\Omega_C$ is the cone-volume measure of a $C$-close set.
\end{theorem}

In the case of cone-volume measures of $C$-full sets or $C$-close sets, the uniqueness question remains open.

After some preparations in Section \ref{sec2}, we prove Theorem \ref{T1} in Section \ref{sec3}. An integral representation for the volume of $C$-coconvex sets is proved in Section \ref{sec4}. This is used in Section \ref{sec5} to introduce mixed volumes of bounded coconvex sets, and in Section \ref{sec6} these mixed volumes are extended to general $C$-coconvex sets. After this, everything is available to prove Theorem \ref{T2}, in Section \ref{sec7}. To prepare the variational proof of Theorem \ref{T3}, we introduce coconvex Wulff shapes in Section \ref{sec8}. Then Theorem \ref{T3} can be proved in Section \ref{sec9}. The remaining three sections deal with cone-volume measures of $C$-close sets and the proofs of  Theorems \ref{T4} and \ref{T5}.

\section{Notation and Preliminaries}\label{sec2}

We fix some notation, and collect what has already been introduced. We work in the $n$-dimensional Euclidean space $\R^n$ ($n\ge 2$), with scalar product $\langle\cdot\,,\cdot\rangle$ and induced norm $\|\cdot\|$. The unit sphere of $\R^n$ is $\Sn:=\{x\in\R^n:\|x\|=1\}$. We use the $k$-dimensional Hausdorff measure $\Ha^k$ on $\Rn$, for $k=n$, which on Lebesgue measurable sets coincides with Lebesgue measure, and for $k=n-1$. For convex bodies or $C$-coconvex sets, the Lebesgue measure, then called the volume, is denoted by $V_n$. Clearly, the volume of a $C$-coconvex set is homogeneous of degree $n$, that is,
$$ V_n(\lambda A)=\lambda^n V_n(A)\quad\mbox{for }\lambda\ge 0.$$

We write hyperplanes and closed halfspaces in the form
\begin{eqnarray*} 
H(u,t) &:=& \{x\in\Rn: \langle u,x\rangle =t\},\\  
H^-(u,t) &:=& \{x\in\Rn: \langle u,x\rangle \le t\},\\
H^+(u,t) &:=& \{x\in\Rn: \langle u,x\rangle \ge t\},
\end{eqnarray*}
with $u\in\Sn$ and $t\in\R$. 

The pointed, closed, convex cone $C\subset\Rn$ with interior points will be kept fixed in the following. Its polar cone is defined by 
$$C^\circ=\{x\in\Rn: \langle x,y\rangle\le 0\mbox{ for all }y\in C\}.$$ 
The set 
$$\Omega_C:= \Sn\cap{\rm int}\,C^\circ$$ 
is an open subset of the unit sphere $\Sn$. The vectors $u\in\Omega_C$ are precisely the unit vectors for which $H(u,0)\cap C =\{o\}$ and, therefore, precisely the unit vectors for which $H^+(u,t)\cap C$ is bounded for $t<0$.

Since the cone $C$ is pointed, we can choose a unit vector $w$ such that $\langle x,w\rangle >0$ for all $x \in C\setminus\{o\}$. The vector $w$ will be fixed; therefore it does not appear in the notation used below. We define the hyperplanes
$$ H_t:= \{x \in \R^n: \langle x,w \rangle =t\}$$
and the closed halfspaces
$$ H^-_t:= \{x \in \R^n: \langle x,w \rangle \le t\},$$
for $t\ge0$. For a subset $M\subseteq C$, we define
$$ M_t:= M \cap H^-_t$$
for $t>0$; thus, $M_t$ is always bounded.

As already mentioned, a set $K$ is called {\em $C$-close} if $K$ is a closed convex subset of $C$ with the property that $C\setminus K$ has positive finite volume. In this case, $C\setminus K$ is called a $C$-coconvex set. A set $K$ is called {\em $C$-full} if $K$ is a closed convex subset of $C$ with the property that $C\setminus K$ is bounded. 

We remark that a $C$-coconvex set $A$ has the property that its boundary inside ${\rm int}\,C$ `can be seen from the origin'. In other words, every ray with endpoint $o$ and passing through an interior point of $C$ meets the boundary of $A$ precisely once. This follows easily from the finiteness of the volume of $A$.

Convergence of $C$-close sets is defined via convergence of compact sets with respect to the Hausdorff metric.

\begin{definition}\label{D2.1}
If $K_j$, $j\in\N_0$, are $C$-close sets, we write
$$ K_j\to K_0 $$
(and say that $(K_j)_{j\in\N}$ converges to $K_0$) if there exists $t_0>0$ such that $K_j\cap C_{t_0}\not=\emptyset$ for all $j\in\N$, and 
$$ \lim_{j\to\infty} (K_j\cap C_t) = K_0\cap C_t \quad\mbox{for all }t\ge t_0,$$
where this means the ordinary convergence of convex bodies with respect to the Hausdorff metric. 
\end{definition}

The support function of a $C$-close set $K$ is defined by
$$h(K,x) = \sup\{\langle x,y\rangle: y\in K\} \quad \mbox{for } x\in {\rm int}\,C^\circ.$$
When convenient, we write $h(K,\cdot)=h_K$. It is easy to see that the supremum is attained, that $h(K,\cdot)$ determines $K$ uniquely, namely by
$$ K= C\cap\bigcap_{u\in\Omega_C} H^-(u,h(K,u)),$$
and that $h(K,\cdot)<0$. For $u\in\Omega_C$, the closed halfspace
$$ H^-(K,u)= \{x\in\Rn: \langle x,u\rangle \le h(K,u)\}$$
is the supporting halfspace of $K$ with outer unit normal vector $u$.

A $C$-coconvex set $A$ and the $C$-close set $C\setminus A$ determine each other uniquely. To make the correspondence more evident, we shall in the following often write $C\setminus A:= A^\bullet$. Thus, sets $A^\bullet$ are always closed and convex, and sets $(A^\bullet)_t$ are in addition bounded.

\section{Proof of Theorem \ref{T1}}\label{sec3}

The following proof of Theorem \ref{T1} has elements from the Kneser--S\"uss proof of the classical Brunn--Minkowski inequality (see, e.g., \cite[pp. 370--371]{Sch14}). 

Let $A_0,A_1$ be $C$-coconvex sets. First we assume that
\begin{equation}\label{1a} 
V_n(A_0)=V_n(A_1)=1.
\end{equation}
Let $0<\lambda<1$ and define 
$$ A_\lambda^\bullet := (1-\lambda)A_0^\bullet +\lambda A_1^\bullet,\qquad A_\lambda:= C\setminus A_\lambda^\bullet =(1-\lambda)A_0\oplus \lambda A_1.$$
In the following, $\nu\in\{0,1\}$. We write
$$ v_\nu(\zeta):= V_{n-1}(A^\bullet_\nu  \cap H_\zeta),\qquad w_\nu(\zeta):= V_n(\A^\bullet_\nu\cap H^-_\zeta)$$
for $\zeta\ge 0$, thus
$$ w_\nu(\zeta)=\int_{\alpha_\nu}^\zeta v_\nu(s)\,\D s,$$
where $\alpha_\nu$ is the number for which $H_{\alpha_\nu}$ supports $A_\nu^\bullet$. On $(\alpha_\nu,\infty)$, the function $v_\nu$ is continuous, hence $w_\nu$ is differentiable and
$$ w_\nu'(\zeta) =v_\nu(\zeta)>0\quad\mbox{for }\alpha_\nu <\zeta <\infty.$$
Let $z_\nu$ be the inverse function of $w_\nu$, then
$$ z_\nu'(\tau) =\frac{1}{v_\nu(z_\nu(\tau))}\quad\mbox{for }0<\tau<\infty.$$
With
$$ D_\nu(\tau):= A_\nu^\bullet\cap H_{z_\nu(\tau)},\quad z_\lambda(\tau):= (1-\lambda)z_0(\tau)+\lambda z_1(\tau),$$
the inclusion
\begin{equation}\label{10} 
A_\lambda^\bullet\cap H_{z_\lambda(\tau)} \supseteq (1-\lambda)D_0(\tau)+ \lambda D_1(\tau)
\end{equation}
holds (trivially). For $\tau>0$ we have
\begin{align}
V_n(A_\nu \cap  H^-_{z_\nu(\tau)}) &= V_n(C\cap H^-_{z_\nu(\tau)}) - V_n(A_\nu^\bullet\cap H^-_{z_\nu(\tau)})\nonumber\\
&= V_n(C\cap H^-_{z_\nu(\tau)}) -\tau,\nonumber\\
V_n(A_\lambda\cap H^-_{z_\lambda(\tau)})&= V_n(C\cap H^-_{z_\lambda(\tau)}) - V_n(A_\lambda^\bullet\cap  H^-_{z_\lambda(\tau)}).\label{3a}
\end{align}
We write
$$ V_n(A_\lambda^\bullet\cap H^-_{z_\lambda(\tau)}) =: f(\tau).$$
Then, with $\alpha_\lambda=(1-\lambda)\alpha_0+\lambda\alpha_1$,
\begin{align*}
f(\tau) &= \int_{\alpha_\lambda}^{z_\lambda(\tau)} V_{n-1}(A^\bullet_\lambda\cap H_\zeta)\,\D \zeta\\
&= \int_0^\tau  V_{n-1}(A^\bullet_\lambda\cap H_{z_\lambda(t)})z_\lambda'(t)\,\D t\\
& \ge \int_0^\tau  V_{n-1}((1-\lambda)D_0(t)+\lambda D_1(t))z_\lambda'(t)\,\D t,
\end{align*}
by (\ref{10}). In the integrand, we use the Brunn--Minkowski inequality in dimension $n-1$ and obtain 
\begin{align}\label{11}
f(\tau) &\ge \int_0^\tau \left[(1-\lambda)v_0(z_0(t))^{\frac{1}{n-1}}+\lambda v_1(z_1(t))^{\frac{1}{n-1}}\right]^{n-1} \left[\frac{1-\lambda}{v_0(z_0(t))}+\frac{\lambda}{v_1(z_1(t))}\right]\D t \nonumber\\
&\ge \tau,
\end{align} 
where the last inequality follows by estimating the integrand according to \cite[p. 371]{Sch14}. 

From (\ref{3a}) we have
$$ V_n(A_\lambda\cap H^-_{z_\lambda(\tau)}) = V_n(C\cap H^-_{z_\lambda(\tau)})-f(\tau),$$
and we intend to let $\tau\to\infty$. Since $C$ is a cone, for $\zeta>0$,
$$ C\cap H^-_\zeta =\zeta C_1 \quad\mbox{with}\quad C_1:= C\cap H^-_1$$
and hence $V_n(C\cap H^-_\zeta)=\zeta^nV_n(C_1)$. Therefore,
$$ V_n(C\cap H^-_{z_\lambda(\tau)}) = [(1-\lambda)z_0(\tau)+\lambda z_1(\tau)]^n V_n(C_1),\qquad V_n(C\cap H^-_{z_\nu(\tau)})= z_\nu(\tau)^nV_n(C_1).$$
This gives
\begin{align*}
& V_n(A_\lambda\cap H^-_{z_\lambda(\tau)})´\\
&= \left[(1-\lambda)V_n(C\cap H^-_{z_0(\tau)})^{\frac{1}{n}} +\lambda V_n(C\cap H^-_{z_1(\tau)})^{\frac{1}{n}}\right]^n-f(\tau)\\
&= \left[(1-\lambda)[V_n(A_0\cap H^-_{z_0(\tau)})+\tau]^{\frac{1}{n}} +\lambda [V_n(A_1\cap H^-_{z_1(\tau)})+\tau]^{\frac{1}{n}}\right]^n-f(\tau)\\
&= \left[(1-\lambda)[b_0(\tau)+\tau]^{\frac{1}{n}} + \lambda[b_1(\tau)+\tau]^{\frac{1}{n}}\right]^n -f(\tau)
\end{align*}
with $b_\nu(\tau)= V_n(A_\nu\cap H^-_{z_\nu(\tau)})$ for $\nu=0,1$. Note that (\ref{1a}) implies
$$ \lim_{\tau\to\infty} b_\nu(\tau)= 1.$$
Using the mean value theorem (for each fixed $\tau$), we can write
$$ (b_1(\tau)+\tau)^{\frac{1}{n}} -(b_0(\tau)+\tau)^{\frac{1}{n}}  = (b_1(\tau)-b_0(\tau))\frac{1}{n}(b(\tau)+\tau)^{\frac{1}{n}-1}$$
with $b(\tau)$ between $b_0(\tau)$ and $b_1(\tau)$, and hence tending to $1$ as $\tau\to\infty$. With $\frac{1}{n}(b(\tau)+\tau)^{\frac{1}{n}-1}=:h(\tau)=O\left(\tau^{\frac{1-n}{n}}\right)$ (as $\tau\to\infty)$, we get
\begin{align*}
& V_n(A_\lambda\cap H^-_{z_\lambda(\tau)})\\ 
&= \left[(1-\lambda)(b_0(\tau)+\tau)^{\frac{1}{n}} +\lambda\left((b_0(\tau)+\tau)^{\frac{1}{n}} + (b_1(\tau)-b_0(\tau))h(\tau)\right)\right]^n-f(\tau)\\
&= \left[(b_0(\tau)+\tau)^{\frac{1}{n}} +\lambda(b_1(\tau)-b_0(\tau))h(\tau))\right]^n-f(\tau)\\
&= b_0(\tau)+\tau-f(\tau) +\sum_{r=1}^n \binom{n}{r} (b_0(\tau)+\tau)^{\frac{n-r}{n}} \left[\lambda(b_1(\tau)-b_0(\tau))\right]^rh(\tau)^r.
\end{align*}
Since  $b_0(\tau)\to 1$, $f(\tau)\ge \tau$, $(b_0(\tau)+\tau)^{\frac{n-r}{n}}h(\tau)^r =O(\tau^{1-r})$, and $b_1(\tau)-b_0(\tau)\to 0$ as $\tau\to\infty$, we conclude that
$$ V_n(A_\lambda) =\lim_{\tau\to\infty} V_n(A_\lambda\cap H^-_{z_\lambda(\tau)}) \le 1.$$
This proves that 
\begin{equation}\label{12}
V_n((1-\lambda)A_0\oplus \lambda A_1)\le 1.
\end{equation}

If there exists a number $\tau_0>0$ for which $f(\tau_0)=\tau_0+\varepsilon$ with $\varepsilon>0$, then, for $\tau>\tau_0$,
\begin{align*}
f(\tau) &= V_n(A_\lambda^\bullet\cap  H^-_{z_\lambda(\tau)})\\
&= \tau_0+\varepsilon+ \int_{\tau_0}^\tau \left[(1-\lambda)v_0(z_0(t))^{\frac{1}{n-1}}+\lambda v_1(z_1(t))^{\frac{1}{n-1}}\right]^{n-1} \left[\frac{1-\lambda}{v_0(z_0(t))}+\frac{\lambda}{v_1(z_1(t))}\right]\D t\\
& \ge \tau_0+\varepsilon+(\tau-\tau_0)= \tau+\varepsilon,
\end{align*}
and as above we obtain that $V_n(A_\lambda)\le 1-\varepsilon$.

Suppose now that (\ref{12}) holds with equality. Then, as just shown, we have $f(\tau)=\tau$ for all $\tau\ge 0$. Thus, we have equality in (\ref{11}) and hence equality in (\ref{10}), for all $\tau\ge 0$. Explicitly, this means that
\begin{equation}\label{13} 
A_\lambda^\bullet \cap H_{z_\lambda(\tau)} = (1-\lambda)(A_0^\bullet \cap H_{z_0(\tau)}) +\lambda(A_1^\bullet \cap H_{z_1(\tau)}) \quad\mbox{for all } \tau\ge 0.
\end{equation}
We claim that this implies
\begin{equation}\label{14} 
A_\lambda^\bullet \cap H^-_{z_\lambda(\tau)} = (1-\lambda)(A_0^\bullet \cap H^-_{z_0(\tau)}) +\lambda(A_1^\bullet \cap H^-_{z_1(\tau)})
\end{equation}
for all $\tau\ge 0$. For the proof, let $x\in A_\lambda^\bullet \cap H^-_{z_\lambda(\tau)}$. Then there is a number $\sigma\in[0,\tau]$ such that $x \in A^\bullet_\lambda \cap H_{z_\lambda(\sigma)}$. By (\ref{13}),
\begin{align*}
x &\in (1-\lambda)(A_0^\bullet\cap H_{z_0(\sigma)}) + \lambda(A_1^\bullet\cap H_{z_1(\sigma)})\\
&\subset (1-\lambda)(A_0^\bullet\cap H^-_{z_0(\tau)}) + \lambda(A_1^\bullet\cap H^-_{z_1(\tau)}),
\end{align*}
since $\sigma\le\tau$ implies $H_{z_\nu(\sigma)} \subset H^-_{z_\nu(\tau)}$. This shows the inclusion $\subseteq$ in (\ref{14}). The inclusion $\supseteq$ is trivial.

To (\ref{14}), we can now apply the Brunn--Minkowski inequality for $n$-dimensional convex bodies and conclude that
$$ V_n(A_\lambda^\bullet \cap H^-_{z_\lambda(\tau)}) \ge \tau.$$
But we know that equality holds here, since equality holds in (\ref{11}), hence the convex bodies $A_0^\bullet \cap H^-_{z_0(\tau)}$ and $A_1^\bullet \cap H^-_{z_1(\tau)}$, which have the same volume, are translates of each other. The translation vector might depend on $\tau$, but in fact, it does not, since for $0<\sigma<\tau$, the body $A^\bullet_\nu\cap H_{z_\nu(\sigma)}$ is the intersection of $A^\bullet_\nu\cap H_{z_\nu(\tau)}$ with a closed halfspace. We conclude that $A^\bullet_1$ is a translate of $A^\bullet_0$, thus there is a vector $v$ with $A^\bullet_0+v=A^\bullet_1\subset C$. Suppose that $v\not= o$. Let $M$ be the set of all points $x\in{\rm int}\,C\cap{\rm bd}\,A^\bullet_0$ for which $x+\lambda v\notin A^\bullet_0$ for $\lambda>0$. The set $\bigcup_{x\in M} (x,x+v]$ is contained in $A_0$ and has infinite Lebesgue measure, a contradiction. Thus, $v=o$ and hence $A^\bullet_0=A^\bullet_1$.

This proves Theorem \ref{T1} under the assumption (\ref{1a}). Now let $A_0,A_1$ be arbitrary $C$-coconvex sets. As mentioned, also the volume of $C$-coconvex sets is homogeneous of degree $n$. Therefore (as in the case of convex bodies, see \cite[p. 370]{Sch14}), we define
$$ \overline A_\nu:= V_n(A_\nu)^{-1/n}A_\nu\enspace\mbox{for }\nu =0,1,\qquad \overline\lambda :=\frac{\lambda V_n(A_1)^{1/n}}{(1-\lambda)V_n(A_0)^{1/n}+\lambda V_n(A_1)^{1/n}}.$$
Then $V_n(\overline A_\nu)=1$ for $\nu=0,1$, hence $V_n((1-\overline \lambda)\overline A_0\oplus\overline\lambda \,\overline A_1)\le 1$, as just proved.  This gives the assertion.

\section{A volume representation}\label{sec4}

The proof of Theorem \ref{T2} requires that we develop the initial steps of a theory of mixed volumes for $C$-coconvex sets. First we derive an integral representation of the volume of $C$-coconvex sets.

Let $A$ be a $C$-coconvex set, and let $u\in\Omega_C$. Since $o\notin A^\bullet$ (because $A\not=\emptyset$), there is a supporting halfspace of $A^\bullet$ with outer normal vector $u$ and not containing $o$. Therefore, the support function $h(A^\bullet,\cdot)$ of $A^\bullet$, defined by $h(A^\bullet,u)=\sup\{\langle x,u\rangle: x\in A^\bullet\}$ for $u\in\Omega_C$, satisfies
$$ -\infty < h(A^\bullet,u)<0\quad\mbox{for }u\in \Omega_C.$$
We set
$$ \overline h(A,u) :=-h(A^\bullet,u)$$
and call the function $\overline h(A,\cdot):\Omega_C\to\R_+$ thus defined the {\em support function} of $A$. The {\em area measure} $\overline S_{n-1}(A,\cdot)$ of $A$ is defined by
$$ \overline S_{n-1}(A,\omega) := S_{n-1}(A^\bullet,\omega)= \Ha^{n-1}(\tau(A^\bullet,\omega))$$
for Borel sets $w\subseteq \Omega_C$. Recall that $\tau(A^\bullet,\omega)$ was defined as the set of boundary points of $A^\bullet$ at which there exists an outer unit normal vector falling in $\omega$. 

The volume of the $C$-coconvex set $A$ has an integral representation similar to that in the case of convex bodies, as stated in the following lemma.

\begin{lemma}\label{L1}
The volume of a $C$-coconvex set $A$ can be represented by
\begin{equation}\label{3.0}
V_n(A)=\frac{1}{n} \int_{\Omega_C} \overline h(A,u)\,\overline S_{n-1}(A,\D u).
\end{equation}
\end{lemma}

\begin{proof}
Recall that $M_t:=M\cap H^-_t$ for $M\subseteq C$, in particular, $C_t=C\cap H^-_t$. We write $(A^\bullet)_t=A^\bullet_t$, and later also $(A^\bullet_i)_t=A^\bullet_{i,t}$.

Let $t>0$ be such that $A^\bullet_t$ has interior points. Let
$$ \omega_t:= \sigma(A^\bullet_t, {\rm int}\,C_t),$$
that is, the spherical image of the set of boundary points of $A^\bullet$ in the interior of $C_t$. Further, let
$$ \eta_t:= \sigma(A^\bullet_t,{\rm bd}\,C)\cap {\rm bd}\,C^\circ.$$
By a standard representation of the volume of convex bodies (formula (5.3) in \cite{Sch14}), we have
$$ V_n(A^\bullet_t) =\frac{1}{n} \int_{\Sn} h(A^\bullet_t,u)\,S_{n-1}(A^\bullet_t,\D u).$$
Here,
$$ \int_{\eta_t} h(A^\bullet_t,u)\,S_{n-1}(A^\bullet_t,\D u)=0,$$
since $u\in\eta_t$ implies $h(A^\bullet_t,u)=0$. We state that
\begin{equation}\label{3.1}
S_{n-1}(A^\bullet_t,\Sn\setminus (\omega_t \cup \eta_t \cup\{w\}))=0.
\end{equation}
For the proof, let $x$ be a boundary point of $A^\bullet_t$ where a vector $u\in \Sn\setminus (\omega_t \cup \eta_t \cup\{w\})$ is attained as outer normal vector. Then $x\notin {\rm int }\,C_t$ and hence $x\in H_t$ or $x\in{\rm bd}\,C$. If $x\in H_t$, then $u\not= w$ implies that $x$ lies in two distinct supporting hyperplanes of $A^\bullet_t$. If $x\in ({\rm bd}\,C)\setminus H_t$, then $u\notin \eta_t$ implies that $x$ lies in two distinct supporting hyperplanes of $A^\bullet_t$. In each case, $x$ is a singular boundary point of $A^\bullet_t$. Now the assertion (\ref{3.1}) follows from \cite[(4.32) and Thm. 2.2.5]{Sch14}. 

As a result, we have 
$$ V_n(A^\bullet_t) =\frac{1}{n} \int_{\omega_t\cup\{w\}} h(A^\bullet_t,u)\,S_{n-1}(A^\bullet_t,\D u).$$
Since
$$ h(A^\bullet_t,w)=t,\qquad S_{n-1}(A^\bullet_t,\{w\})= V_{n-1}(A^\bullet\cap H_t),$$
 we obtain
$$ V_n(A^\bullet_t)= -\frac{1}{n} \int_{\omega_t} \overline h(A,u)\,\overline S_{n-1}(A,\D u) +\frac{1}{n}tV_{n-1}(A^\bullet\cap H_t),$$
by the definition of $\overline h(A,\cdot)$ and $\overline S_{n-1}(A,\cdot)$. Writing
$$ B(t):= {\rm conv}((A^\bullet\cap H_t)\cup\{o\})\setminus A^\bullet_t,$$
we have
$$ V_n(B(t))=\frac{1}{n}tV_{n-1}(A^\bullet\cap H_t)-V_n(A^\bullet_t)$$
and thus
$$ V_n(B(t))=\frac{1}{n} \int_{\omega_t} \overline h(A,u)\,\overline S_{n-1}(A,\D u). $$
On the other hand, writing
$$ q(t):= V_{n-1}(C\cap H_t)-V_{n-1}(A^\bullet\cap H_t),$$
we get
$$ V_n(A_t) = V_n(B(t))+\frac{1}{n}tq(t) = \frac{1}{n} \int_{\omega_t} \overline h(A,u)\,\overline S_{n-1}(A,\D u)+\frac{1}{n}tq(t).$$

Given $\varepsilon>0$, to each $t_0>0$ there exists $t\ge t_0$ with $tq(t)<\varepsilon$. Otherwise, there would exist $t_0$ with $tq(t)\ge\varepsilon$ for $t\ge t_0$ and hence $\int_{t_0}^\infty q(t)\D t =\infty$, which yields $V_n(A)=\infty$, a contradiction. Therefore, we can choose an increasing sequence $(t_i)_{i\in\N}$ with $t_i\to\infty$ for $i\to\infty$ such that $t_iq(t_i)\to 0$. From
$$ V_n(A_{t_i}) = \frac{1}{n}\int_{\omega_{t_i}} \overline h(A,u)\,\overline S_{n-1}(A,\D u) +\frac{1}{n}t_iq(t_i)$$
and $\omega_{t_i}\uparrow \Omega_C$ we then obtain 
$$ V_n(A) = \frac{1}{n}\int_{\Omega_C} \overline h(A,u)\,\overline S_{n-1}(A,\D u),$$
as stated.
\end{proof}

\section{Mixed volumes of bounded $C$-coconvex sets}\label{sec5}

First we introduce, in this section, mixed volumes and their representations for bounded coconvex sets. Let $A$ be a bounded $C$-coconvex set. Then $A\subset {\rm int}\,H^-_t$ for all sufficiently large $t$. For bounded $C$-coconvex sets $A_1,\dots,A_{n-1}$, we define their {\em mixed area measure} by
$$ \overline S(A_1,\dots,A_{n-1},\omega) = S(A^\bullet_{1,t},\dots,A^\bullet_{n-1,t},\omega)$$
for Borel sets $\omega\subseteq\Omega_C$, where $t$ is chosen sufficiently large. Here $S(A^\bullet_{1,t},\dots,A^\bullet_{n-1,t},\cdot)$ is the usual mixed area measure of the convex bodies $A^\bullet_{1,t},\dots,A^\bullet_{n-1,t}$ (see \cite[Sect. 5.1]{Sch14}). Clearly, the definition does not depend on $t$. It should be noted that the mixed area measure of bounded $C$-coconvex sets is only defined on $\Omega_C$, and it is finite. For bounded $C$-coconvex sets $A_1,\dots,A_n$, we define  their {\em mixed volume} by
\begin{equation}\label{5.0}
\overline V(A_1,\dots,A_n)=\frac{1}{n}\int_{\Omega_C} \overline h(A_1,u)\,\overline S(A_2,\dots,A_n,\D u).
\end{equation}

\begin{lemma}\label{L2}
The mixed volume $ \overline V(A_1,\dots,A_n)$ is symmetric in $A_1,\dots,A_n$.
\end{lemma}

\begin{proof}
We choose $t$ so large that $A_i\subset H^-_t$ for $i=1,\dots,n$. The mixed volume of the convex bodies $A^\bullet_{1,t},\dots,A^\bullet_{n,t}$ is given by
$$ V(A^\bullet_{1,t},\dots,A^\bullet_{n,t}) =\frac{1}{n} \int_{\Sn} h(A^\bullet_{1,t},u)\,S(A^\bullet_{2,t},\dots,A^\bullet_{n,t},\D u).$$
The sphere $\Sn$ is the disjoint union of the sets
$$ \Omega_C,\; \Sn\cap {\rm bd}\,C^\circ,\; \{w\}, \mbox{ and the remaining set } \omega_0.$$
For $u\in \Sn\cap {\rm bd}\,C^\circ$, we have $h(A^\bullet_{1,t},u)=0$. Since for each body $A^\bullet_{i,t}$ the support set with outer normal vector $w$ is equal to $C\cap H_t$, we get $S_{n-1}(A^\bullet_{i,t},\{w\})=V_{n-1}(C\cap H_t)$ for $i=2,\dots,n$ and thus, by \cite[(5.18)]{Sch14},
$$ S(A^\bullet_{2,t},\dots, A^\bullet_{n,t},\{w\})=V_{n-1}(C\cap H_t).$$
Therefore,
$$\frac{1}{n} \int_{\{w\}} h(A^\bullet_{1,t},u)\,S(A^\bullet_{2,t},\dots,A^\bullet_{n,t},\D u) = \frac{1}{n}tV_{n-1}(C\cap H_t)= V_n(C_t).$$
Further, we have
\begin{equation}\label{5.a} 
S(A^\bullet_{2,t},\dots,A^\bullet_{n,t},\omega_0)=0,
\end{equation}
since for $\lambda_2,\dots,\lambda_n\ge 0$, the convex body $\lambda_2A^\bullet_{2,t}+\dots+\lambda_n A^\bullet_{n,t}$ has the property that any of its points at which some $u\in\omega_0$ is an outer normal vector, is a singular point. Equation (\ref{5.a}) then follows from \cite[(5.21) and Thm. 2.2.5]{Sch14}. As a result, we obtain
\begin{align*}
V(A^\bullet_{1,t},\dots,A^\bullet_{n,t}) &= V_n(C_t)+ \frac{1}{n} \int_{\Omega_C} h(A^\bullet_{1,t},u)\,S(A^\bullet_{2,t},\dots,A^\bullet_{n,t},\D u)\\
&= V_n(C_t) - \frac{1}{n} \int_{\Omega_C} \overline h(A_1,u)\,\overline S(A_2,\dots,A_n,\D u)\\
&= V_n(C_t)-\overline V(A_1,\dots,A_n).
\end{align*}
Since $V(A^\bullet_{1,t},\dots,A^\bullet_{n,t})$ is symmetric in its arguments, also $\overline V(A_1,\dots,A_n)$ is symmetric in its arguments.
\end{proof}

Now let $A_1,\dots,A_m$, with $m\in \N$, be bounded $C$-coconvex sets, and choose $t>0$ with $A_i\subset C_t$ for $i=1,\dots,m$. By (\ref{3.0}), for $\lambda_1,\dots,\lambda_m\ge 0$,
\begin{align*}
& 
V_n(\lambda_1A_1\oplus\dots\oplus \lambda_mA_m)\\
&= \frac{1}{n} \int_{\Omega_C} \overline h(\lambda_1A_1\oplus\dots\oplus \lambda_mA_m,u)\,\overline S_{n-1}(\lambda_1A_1\oplus\dots\oplus \lambda_mA_m,\D u).
\end{align*}
Here, for $u\in\Omega_C$,
\begin{align}
\overline h(\lambda_1A_1\oplus\dots\oplus \lambda_mA_m,u) & =- h((\lambda_1A_1\oplus\dots\oplus \lambda_mA_m)^\bullet,u) \nonumber\\
&= - h(\lambda_1A_1^\bullet+\dots+ \lambda_mA_m^\bullet,u)\nonumber\\
&=-[\lambda_1h(A_1^\bullet,u)+\dots+ \lambda_m(A_m^\bullet,u)] \label{5.10}
\end{align}
and, for $t$ sufficiently large and Borel sets $\omega\subseteq\Omega_C$,
\begin{align}
&\overline S_{n-1}(\lambda_1 A_1 \oplus\dots\oplus \lambda_m A_m,\omega)\nonumber\\
&= S_{n-1}(\lambda_1 A_{t,1}^\bullet +\dots+ \lambda_m A_{t,m}^\bullet,\omega)\label{5.11}\\
&= \sum_{i_1,\dots,i_{n-1}=1}^m \lambda_{i_1}\cdots \lambda_{i_{n-1}}S(A_{t,i_1}^\bullet,\dots,A_{t,i_{n-1}}^\bullet,\omega),\nonumber\\
&= \sum_{i_1,\dots,i_{n-1}=1}^m \lambda_{i_1}\cdots \lambda_{i_{n-1}}\overline S(A_{i_1},\dots,A_{i_{n-1}},\omega),\nonumber
\end{align}
by \cite[(5.18)]{Sch14}. Using Lemma \ref{L2}, we conclude that
\begin{equation}\label{7} 
V_n(\lambda_1 A_1 \oplus\dots\oplus \lambda_m A_m) = \sum_{i_1,\dots,i_n=1}^m \lambda_{i_1}\cdots\lambda_{i_n} \overline V(A_{i_1},\dots,A_{i_n}),
\end{equation}
in analogy to  \cite[(5.17)]{Sch14}.

\section{Mixed volumes of general $C$-coconvex sets}\label{sec6}

We extend the mixed volumes to not necessarily bounded $C$-coconvex sets. For this, we use approximation by mixed volumes of bounded $C$-coconvex sets.

Let $\omega\subset\Omega_C$ be an open subset whose closure (in $\Sn$) is contained in $\Omega_C$. Let $A$ be a $C$-coconvex set, so that $A^\bullet =C\setminus A$ is closed and convex. We define
$$ A^\bullet_{(\omega)} := C\cap \bigcap_{u\in\omega} H^-(A^\bullet,u),\qquad A_{(\omega)}:= C\setminus A^\bullet_{(\omega)},$$
where $H^-(A^\bullet, u)$ denotes the supporting halfspace of the closed convex set $A^\bullet$ with outer normal vector $u$. We claim that $A_{(\omega)}$ is bounded. For the proof, we note that the set $\omega$, whose closure, ${\rm clos}\,\omega$, is contained in $\Omega_C$, has a positive distance from the boundary of $\Omega_C$ (relative to $\Sn$). Therefore, there is a number $a_0>0$ such that
\begin{equation}\label{6.0}
\langle x,u\rangle \le -a_0\quad\mbox{for $x\in C$ with $\|x\|=1$ and $u\in\omega$.}
\end{equation}
Let $x\in A_{(\omega)}$. Then there is some $u\in\omega$ with $x\notin H^-(A^\bullet, u)$, hence with $\langle x,u\rangle > h(A^\bullet,u)$. Since $\langle x,u\rangle \le -a_0\|x\|$ by (\ref{6.0}), we obtain
$$ \|x\| \le\frac{1}{a_0} \max\{-h(A^\bullet,u):u\in{\rm clos}\,\omega\}.$$
Thus, $A_{(\omega)}$ is a bounded $C$-coconvex set. 

With $A$ and $\omega$ as above, we associate another set, namely
$$ A[\omega]:= \bigcup_{x\in \tau(A^\bullet,\omega)\cap {\rm int}\,C} (o,x),$$
where $(o,x)$ denotes the open line segment with endpoints $o$ and $x$. We choose an increasing sequence $(\omega_j)_{j\in\N}$ of open subsets of $\Omega_C$ with closures in $\Omega_C$ and with $\bigcup_{j\in\N} \omega_j=\Omega_C$. Then
\begin{equation}\label{6.4a}
A[\omega_j] \uparrow {\rm int}\,A \qquad \mbox{as }j\to\infty.
\end{equation}
In fact, that the set sequence is increasing, follows from the definition. Let $y\in {\rm int}\,A$. Then there is a boundary point $x$ of $A^\bullet$ with $y\in (o,x)$. Let $u$ be an outer unit normal vector of $A^\bullet$ at $x$. Then $u\in\Omega_C$, hence $u\in\omega_j$ for some $j$. For this $j$, we have $y\in A[\omega_j]$. 

\begin{lemma}\label{L6.1}
If $A_1,\dots,A_n$ are $C$-coconvex sets and $\lambda_1,\dots,\lambda_n\ge 0$, then
\begin{equation}\label{6.10}
\lim_{j\to\infty} V_n(\lambda_1 A_{1(\omega_j)}\oplus\dots\oplus \lambda_n A_{n(\omega_j)}) = V_n(\lambda_1 A_1\oplus\dots\oplus \lambda_n A_n).
\end{equation}
\end{lemma}

\begin{proof}
We state that
\begin{equation}\label{6.11} 
(\lambda_1 A_1\oplus\dots\oplus \lambda_n A_n)[\omega_j] \subseteq \lambda_1 A_{1(\omega_j)}\oplus\dots\oplus \lambda_n A_{n(\omega_j)} \subseteq \lambda_1 A_1\oplus\dots\oplus \lambda_n A_n.
\end{equation}
For the proof of the first inclusion, let $y \in (\lambda_1 A_1 \oplus \dots \oplus \lambda_n A_n)[\omega_j] $. Then there exists a point $x \in \tau(\lambda_1 A^\bullet_1 +\dots + \lambda_n A^\bullet_n,\omega_j)\cap {\rm int}\,C$ with $y \in (o,x)$. Let $u \in \omega_j$ be an outer unit normal vector of $\lambda_1 A^\bullet_1 + \dots + \lambda_n A^\bullet_n$ at $x$. Denoting by $F(K,u)$ the support set of a closed convex set $K$ with outer normal vector $u$, we have (by \cite[Thm. 1.7.5]{Sch14})
$$ F(\lambda_1 A^\bullet_1 +\dots + \lambda_n A^\bullet_n,u) =\lambda_1 F(A^\bullet_1,u)+\dots + \lambda_n F(A^\bullet_n,u),$$
hence there are points $x_i \in F(A^\bullet_i,u)$ ($i=1,\dots,n$) with $x=\lambda_1 x_1+\dots+\lambda_n x_n$. We have $x_i \in A^\bullet_{i(\omega_j)}$, hence $x\in \lambda_1 A_{1(\omega_j)}\oplus\dots\oplus \lambda_n A_{n(\omega_j)}$. This proves the first inclusion of (\ref{6.11}). The second inclusion follows immediately from the definitions. From (\ref{6.11}) and (\ref{6.4a}) we obtain 
$$ \lambda_1 A_{1(\omega_j)}\oplus\dots\oplus \lambda_n A_{n(\omega_j)}  \uparrow {\rm int}\,(\lambda_1 A_1\oplus\dots\oplus \lambda_n A_n)  \qquad \mbox{as }j\to\infty,$$
from which the assertion (\ref{6.10}) follows.
\end{proof}

For the bounded $C$-coconvex sets $A_{1(\omega_j)},\dots, A_{n(\omega_j)}$ we have from (\ref{7}) that
$$ V_n(\lambda_1 A_{1(\omega_j)}\oplus\dots\oplus \lambda_n A_{n(\omega_j)}) = \sum_{i_1,\dots,i_n=1}^n \lambda_{i_1}\cdots\lambda_{i_n} \overline V(A_{i_1(\omega_j)},\cdots, A_{i_n(\omega_j)}).$$
By Lemma \ref{L6.1}, the left side converges, for $j\to\infty$, to $V_n(\lambda_1 A_1\oplus\dots\oplus \lambda_n A_n)$. Since this holds for all $\lambda_1,\dots,\lambda_n\ge 0$, we can conclude that the limit
$$ \lim_{j\to\infty} \overline V(A_{i_1(\omega_j)},\dots,A_{i_n(\omega_j)}) =:\overline V(A_{i_1},\dots,A_{i_n})$$
exists and that
\begin{equation}\label{3.5} 
V_n(\lambda_1 A_1 \oplus\dots\oplus \lambda_n A_n) =  \sum_{i_1,\dots,i_n=1}^n \lambda_{i_1}\cdots\lambda_{i_n} \overline V(A_{i_1},\dots,A_{i_n}).
\end{equation}
We call $\overline V(A_1,\dots,A_n)$ the {\em mixed volume} of the $C$-coconvex sets $A_1,\dots,A_n$.

For this mixed volume, we shall now establish an integral representation. To that end, we note first that the support functions of $A^\bullet$ and $A^\bullet_{(\omega_j)}$ satisfy
\begin{equation}\label{13a} 
h(A^\bullet,u)= h(A^\bullet_{(\omega_j)},u) \quad\mbox{for }u\in\omega_j.
\end{equation}
Since $\omega_j$ is open, then for $u\in\omega_j$ the support functions of $A^\bullet$ and $A^\bullet_{(\omega_j)}$ coincide in a neighborhood of $u$. By \cite[Thm. 1.7.2]{Sch14}, the support sets of $A^\bullet$ and $A^\bullet_{(\omega_j)}$ with outer normal vector $u$ are the same. It follows that $\tau(A^\bullet,\omega_j) =\tau(A^\bullet_{(\omega_j)},\omega_j)$ and, therefore, that also
\begin{equation}\label{6.1}  
S_{n-1}(A^\bullet,\cdot) = S_{n-1}(A^\bullet_{(\omega_j)},\cdot) \quad\mbox{on }\omega_j.
\end{equation}

More generally, if $A_1,\dots,A_{n-1}$ are $C$-coconvex sets, we can define their mixed area measure by
\begin{equation}\label{6.4} 
S(A^\bullet_1,\dots,A^\bullet_{n-1},\cdot) = S(A^\bullet_{1(\omega_j)},\dots,A^\bullet_{n-1(\omega_j)},\cdot) \quad\mbox{on }\omega_j,
\end{equation}
for $j\in \N$. Since $\omega_j\uparrow\Omega_C$, this yields a Borel measure on all of $\Omega_C$. It need not be finite. Then we define
$$ \overline S(A_1,\dots,A_{n-1},\cdot) := S(A^\bullet_1,\dots,A^\bullet_{n-1},\cdot).$$

By Lemma \ref{L1}, (\ref{13a}) and (\ref{6.1}) we have
\begin{eqnarray*}
&& V_n(A_{(\omega_j)})\\
&&= \frac{1}{n} \int_{\omega_j} \overline h(A_{(\omega_j)},u)\, \overline S_{n-1}(A_{(\omega_j)},\D u) +  \frac{1}{n} \int_{\Omega_C\setminus\omega_j} \overline h(A_{(\omega_j)},u)\, \overline S_{n-1}(A_{(\omega_j)},\D u)\\
&&=\frac{1}{n} \int_{\omega_j} \overline h(A,u)\, \overline S_{n-1}(A,\D u) +  \frac{1}{n} \int_{\Omega_C\setminus\omega_j} \overline h(A_{(\omega_j)},u)\, \overline S_{n-1}(A_{(\omega_j)},\D u).
\end{eqnarray*}
From $A_{(\omega_j)}\uparrow A$ we get
\begin{equation}\label{6.2} 
\lim_{j\to\infty} V_n(A_{(\omega_j)})=V_n(A),
\end{equation}
and $\omega_j\uparrow \Omega_C$ gives
$$  \lim_{j\to\infty} \frac{1}{n} \int_{\omega_j} \overline h(A,u)\, \overline S_{n-1}(A,\D u) =\frac{1}{n} \int_{\Omega_C} \overline h(A,u)\, \overline S_{n-1}(A,\D u) =V_n(A).$$
It follows that
\begin{equation}\label{6.3}
\lim_{j\to\infty} \int_{\Omega_C\setminus\omega_j} \overline h(A_{(\omega_j)},u)\, \overline S_{n-1}(A_{(\omega_j)},\D u)=0.
\end{equation}

From (\ref{5.0}) and using (\ref{6.1}) and (\ref{6.4}), we get
\begin{eqnarray}
&& \overline V(A_{1(\omega_j)},\dots,A_{n(\omega_j)})\nonumber\\
&&= \frac{1}{n} \int_{\omega_j} \overline h(A_{1(\omega_j)},u)\,\overline S(A_{2(\omega_j)},\dots,A_{n(\omega_j)},\D u)\nonumber\\
&& \hspace{4mm} + \frac{1}{n} \int_{\Omega_C\setminus \omega_j} \overline h(A_{1(\omega_j)},u)\,\overline S(A_{2(\omega_j)},\dots,A_{n(\omega_j)},\D u)\nonumber\\
&&= \frac{1}{n} \int_{\omega_j} \overline h(A_{1},u)\,\overline S(A_2,\dots,A_{n},\D u)\nonumber\\
&& \hspace{4mm} + \frac{1}{n} \int_{\Omega_C\setminus \omega_j} \overline h(A_{1(\omega_j)},u)\,\overline S(A_{2(\omega_j)},\dots,A_{n(\omega_j)},\D u) \label{6.12}
\end{eqnarray}
Writing $A:=A_1\oplus\dots\oplus A_n$ we have the trivial estimates
$$ \overline h(A_{1(\omega_j)},u) \le \overline h(A_{(\omega_j)},u),\qquad \overline S(A_{2(\omega_j)},\dots,A_{n(\omega_j)},\cdot) \le \overline S_{n-1}(A_{(\omega_j)},\cdot).$$
Hence, the term (\ref{6.12}) can be estimated by
\begin{eqnarray*}
&& \frac{1}{n} \int_{\Omega_C\setminus \omega_j} \overline h(A_{1(\omega_j)},u)\,\overline S(A_{2(\omega_j)},\dots,A_{n(\omega_j)},\D u)\\
&& \le \frac{1}{n} \int_{\Omega_C\setminus \omega_j} \overline h(A_{(\omega_j)},u)\,\overline S_{n-1}(A_{(\omega_j)},\D u),
\end{eqnarray*}
and by (\ref{6.3}) this tends to zero for $j\to\infty$. We conclude that
\begin{equation}\label{6.5}  
\overline V(A_{1},\dots,A_{n}) =  \frac{1}{n} \int_{\Omega_C} \overline h(A_{1},u)\,\overline S(A_{2},\dots,A_{n},\D u).
\end{equation}

\section{Proof of Theorem \ref{T2}}\label{sec7}

Theorem \ref{T1} together with the polynomial expansion (\ref{3.5}) now allows similar conclusions as in the case of convex bodies. Let $A_0,A_1$ be $C$-coconvex sets, and write $A_\lambda= (1-\lambda)A_0\oplus\lambda A_1$ for $0\le \lambda\le 1$. A special case of (\ref{3.5}) reads
$$ V_n(A_\lambda) =\sum_{i=0}^n\binom{n}{i}(1-\lambda)^{n-i}\lambda^i\overline V(\underbrace{A_0,\dots,A_0}_{n-i},\underbrace{A_1,\dots,A_1}_i).$$
The function $f$ defined by $f(\lambda)=V_n(A_\lambda)^{1/n}-(1-\lambda)V_n(A_0)^{1/n}-\lambda V_n(A_1)^{1/n}$ for $0\le \lambda\le 1$ is convex, as follows from Theorem \ref{T1} and a similar argument as in the case of convex bodies (see \cite[pp. 369--370]{Sch14}). Also as in the convex body case (see \cite[p. 382]{Sch14}), one obtains the counterpart to Minkowski's first inequality, namely
\begin{equation}\label{8}
\overline V(A_0,\dots,A_0,A_1)^n\le V_n(A_0)^{n-1}V_n(A_1),
\end{equation}
with equality if and only if $A_0=\alpha A_1$ with some $\alpha>0$.

Now, under the assumptions of Theorem \ref{T2} there are $C$-coconvex sets $A_0,A_1$ with $\overline S_{n-1}(A_0,\cdot)= \overline S_{n-1}(A_1,\cdot)$. By (\ref{6.5}), 
$$ \overline V(A_0,\dots,A_0,A_1)= \frac{1}{n} \int_{\Omega_C} \overline h(A_1,u)\,\overline S_{n-1}(A_0,\D u).$$
Therefore, the assumption gives $\overline V(A_0,\dots,A_0,A_1) = V_n(A_1)$. Since $A_0$ and $A_1$ can be interchanged, we also have $\overline V(A_1,\dots,A_1,A_0) = V_n(A_0)$, hence multiplication gives 
$$ \overline V(A_0,\dots,A_0,A_1) \overline V(A_1,\dots,A_1,A_0) = V_n(A_0)V_n(A_1).$$
On the other hand, from (\ref{8}) we get 
$$\overline V(A_0,\dots,A_0,A_1) \overline V(A_1,\dots,A_1,A_0)\le V_n(A_0)V_n(A_1).$$ 
Thus, equality holds here, and hence in (\ref{8}), which implies that $A_0=\alpha A_1$ with $\alpha>0$. Since $\overline S_{n-1}(A_0,\cdot)= \overline S_{n-1}(A_1,\cdot)$, we have $\alpha=1$. This proves Theorem \ref{T2}.

\section{Coconvex Wulff shapes}\label{sec8}

In the theory of convex bodies, the Wulff shape, or Aleksandrov body, is a useful concept, in particular in connection with Aleksandrov's variational lemma. This is Lemma IV in Aleksandrov's classical work \cite{Ale38a}; it was reproduced in \cite[Sect. 7.5]{Sch14} and was the template for several later generalizations and applications. Here we carry this concept over to the coconvex setting. Also the purpose is the same: to provide a variational argument which is of use in the proofs of Theorems \ref{T3} and \ref{T4}. In establishing some properties of the Wulff shape, we can argue similarly as in \cite{Ale38a}, but because of some essential differences, we carry out the details.

In this section, $\omega\subset\Omega_C$ is a nonempty, compact set. The set $\omega$ has a positive distance from the boundary of $\Omega_C$ (relative to $\Sn$). Therefore, there is a number $a_0>0$ such that
\begin{equation}\label{1}
\langle x,u\rangle \le -a_0\quad\mbox{for $x\in C$ with $\|x\|=1$ and $u\in\omega$}.
\end{equation}
We say that a closed convex set $K$ is {\em $C$-determined by} $\omega$ if
$$ K= C\cap\bigcap_{u\in\omega} H^-(K,u).$$
By $\K(C,\omega)$ we denote the set of all closed convex sets that are $C$-determined by $\omega$. If $K\in\K(C,\omega)$ then $C\setminus K$ is bounded. In fact, let $x\in C\setminus K$. Then there is some $u\in\omega$ with $x\notin H^-(u,h_K(u))$, hence with $\langle  x,u\rangle > h_K(u)$. Since $\langle x,u\rangle \le -a_0\|x\|$ by (\ref{1}), we obtain
\begin{equation}\label{3}
\|x\|\le \frac{1}{a_0}\max\{ -h_K(u):u\in\omega\}.
\end{equation}
The maximum exists since $h_K$ is continuous and $\omega$ is compact.

Let $f:\omega\to\R$ be a positive, continuous function. The closed convex set
\begin{equation}\label{97} 
K:= C\cap\bigcap_{u\in\omega} H^-(u,-f(u)) 
\end{equation}
is called the {\em Wulff shape} associated with $(C,\omega,f)$. It follows from the definition that
\begin{equation}\label{2}
-h(K,u) \ge f(u)\quad\mbox{for }u\in\omega
\end{equation} 
and that $K$ is $C$-determined by $\omega$. In particular, a Wulff shape is a $C$-full set.

\begin{lemma}\label{L7.5.2}
If $K_j$ is the Wulff shape associated with $(C,\omega,f_j)$, for $j\in\N_0$, and if $(f_j)_{j\in\N}$ converges uniformly (on $\omega$) to $f_0$, then $K_j\to K_0$ (in the sense of Definition $\ref{D2.1}$).
\end{lemma}

\begin{proof} Because convergence in the sense of Definition \ref{D2.1} is equivalent to convergence of suitable sequences of convex bodies, we can use the convergene criterion \cite[Thm. 1.8.8]{Sch14}. First let $x\in{\rm int}\,K_0$. Then there is some $\varepsilon>0$ with $\langle x,u\rangle \le -f_0(u)-\varepsilon$ for all $u\in\omega$. Since $f_j\to f_0$ uniformly, there is some $j_0$ with $|f_j(u)-f_0(u)|<\varepsilon$ for all $j\ge j_0$ and all $u\in\omega$. Therefore, $\langle x,u\rangle\le -f_j(u)$ for all $j\ge j_0$ and all $u\in\omega$. This shows that $x\in K_j$ for $j\ge j_0$. If now $x_0\in K_0$ (but not necessarily $x_0\in {\rm int}\,K_0$), we can choose a sequence $(x_i)_{i\in\N}$ in ${\rm int}\,K_0$ with $x_i\to x_0$. Using the preceding observation, it is easy to construct a sequence $(y_j)_{j\in\N}$ with $y_j\in K_j$ and $y_j\to x_0$. Conversely, suppose that $x_{i_j}\in K_{i_j}$ and $x_{i_j}\to x_0$. Then $\langle x_{i_j}, u\rangle \le -f_{i_j}(u)$ for all $u\in\omega$. It follows that $\langle x_0,u\rangle \le -f(u)$ for all $u\in\omega$ and thus $x_0\in K_0$. Now it follows from \cite[Thm. 1.8.8]{Sch14} that $\lim_{j\to\infty} (K_j\cap C_t)= K_0\cap C_t$ for all sufficiently large $t>0$.
\end{proof}

\begin{lemma}\label{L3.2}
If $(K_j)_{j\in\N}$ is a sequence in $\K(C,\omega)$ such that $K_j\to K_0$ for some $C$-full set $K_0$, then $K_0\in \K(C,\omega)$. 

\end{lemma}

\begin{proof}
From $K_j\to K_0$ it follows that $h_{K_j}\to h_{K_0}$ uniformly on $\omega$, hence, by Lemma \ref{L7.5.2}, $K_j\to K$, where $K$ is the Wulff shape associated with $(C,\omega,-h_K)$. But $K=K_0$, hence $K_0\in\K(C,\omega)$. 
\end{proof}

In the following, let $K$ be the Wulff shape associated with $(C,\omega,f)$. First we state that
\begin{equation}\label{98}
S_{n-1}(K,\Omega_C\setminus\omega)=0.
\end{equation}
For the proof, let $x\in {\rm bd}\,K \cap {\rm int}\,C$. Then there exists a vector $u\in\omega$ such that $x\in H(u,-f(u))$, since otherwise $\langle x,u\rangle <-f(u)$ for all $u\in\omega$, and since $\langle x,\cdot\rangle$ and $f$ are continuous on the closed set $\omega$, there would exist a number $\varepsilon>0$ with $\langle x,u\rangle \le -f(u)-\varepsilon$ for all $u\in\omega$, hence $x\in{\rm int}\,K$, a contradiction. This shows that $x\in H(K,u)$ for some $u\in\omega$.

Now let $v\in\Omega_C\setminus\omega$ and $x\in K\cap H(K,v)$. If $x\in{\rm int}\,C$, then $x\in H(K,u)$ for some $u\in\omega$, as just shown. Hence, $x$ is a singular point of $K$. If $x\in{\rm bd}\,C$, then the point $x$ lies also in a supporting hyperplane of $C$, with a normal vector in ${\rm bd}\,\Omega_C$ and thus different from $v$, and hence is again a singular point of $K$. The assertion (\ref{98}) now follows from \cite[Thm. 2.2.5]{Sch14}.

If $v\in \omega$ is such that $h(K,v) \not= -f(v)$, then any point $x\in K\cap H(K,v)$ lies in $H(K,v)$ and in some distinct supporting hyperplane $H(u,-f(u))$ with $u\in\omega$, hence $x$ is singular. This shows that
\begin{equation}\label{100}
S_{n-1}(K, \{v\in\omega: -h(K,v) \not= f(v)\})=0.
\end{equation}
Now we can deduce from (\ref{3.0}), (\ref{98}) and (\ref{100}) that
\begin{equation}\label{99}
V_n(C\setminus K)=\frac{1}{n} \int_\omega f(u)\,S_{n-1}(K,\D u).
\end{equation}
We define
$$ V(f):= V_n(C\setminus K),$$
where $K$ is the Wulff shape associated with $(C,\omega,f)$. The function $V$ is continuous: if $f_j\to f$ uniformly on $\omega$, then $V(f_j)\to V(f)$.

We assume now that a continuous function $G:[-\varepsilon,\varepsilon]\times\omega\to\R$, for some $\varepsilon>0$, with $G(0,\cdot)>0$  is given and that there is a continuous function $g:\omega\to\R$ such that
\begin{equation}\label{101}
\lim_{\tau\downarrow 0}\frac{G(\tau,\cdot)-G(0,\cdot)}{\tau}= g\quad\mbox{uniformly on }\omega.
\end{equation}
For all sufficiently small $|\tau|$, the function $G(\tau,\cdot)$ is positive, hence $V(G(\tau,\cdot))$ is defined.

\begin{lemma}\label{L7.5.3}
Let $G$ be as above, and let $K_0$ be the Wulff shape associated with $(C,\omega,G(0,\cdot))$. Then
\begin{equation}\label{102}
\lim_{\tau\downarrow 0} \frac{V(G(\tau,\cdot))-V(G(0,\cdot))}{\tau} = \int_\omega g(u)\,S_{n-1}(K_0,\D u).
\end{equation}
The same assertion holds if in $(\ref{101})$ and $(\ref{102})$ the one-sided limit $\lim_{\tau\downarrow 0}$ is replaced by $\lim_{\tau\uparrow 0}$ or $\lim_{\tau\to 0}$.
\end{lemma}

\begin{proof}
For sufficiently small $\tau>0$, let $K_\tau$ be the Wulff shape associated with the triple $(C,\omega,G(\tau,\cdot))$. Let $A_\tau:= C\setminus K_\tau$. We need
\begin{eqnarray*}
V_0 &:=& V(G(0,\cdot)) = V_n(A_0),\\
V_1(\tau) &:=& \overline V(A_\tau,A_0,\dots,A_0),\\
V_{n-1}(\tau) &:=& \overline V(A_\tau,\dots,A_\tau,A_0)\\
V_n(\tau) &:=& V_n(A_\tau).
\end{eqnarray*}
By (\ref{2}), we have $-h(K_\tau,u) \ge G(\tau,u)$ for $u\in\omega$. Together with (\ref{98}), this yields
$$ \int_{\Omega_C}- h(K_\tau,u)\, S_{n-1}(K_0,\D u) \ge \int_\omega G(\tau,u)\, S_{n-1}(K_0,\D u),$$
hence, by (\ref{5.0}) and (\ref{99}),
$$ V_1(\tau)-V_0 \ge \frac{1}{n} \int_\omega ´[G(\tau,u)-G(0,u)]\,S_{n-1}(K_0,\D u).$$
From (\ref{101}) we deduce that
\begin{equation}\label{103}
\liminf_{\tau\downarrow 0} \frac{V_1(\tau)-V_0}{\tau} \ge \frac{1}{n} \int_\omega g(u)\,  S_{n-1}(K_0,\D u).
\end{equation}
From (\ref{99}), applied to $G(\tau,\cdot)$,
$$ \frac{1}{n} \int_\omega G(\tau,u)\, S_{n-1}(K_\tau,\D u) = V_n(\tau).$$
Since $-h(K_0,u) \ge G(0,u)$ for $u\in\omega$, we have, using (\ref{98}) again,
$$ \frac{1}{n} \int_\omega G(0,u)\, S_{n-1}(K_\tau,\D u) \le \frac{1}{n} \int_{\Omega_C} - h(K_0,u)\, S_{n-1}(K_\tau,\D u)=V_{n-1}(\tau).$$
Subtraction gives
\begin{equation}\label{13b} 
\frac{1}{n} \int_\omega [G(\tau,u)-G(0,u)]\, S_{n-1}(K_\tau,\D u) \ge V_n(\tau)-V_{n-1}(\tau).
\end{equation} 

For $\tau\to 0$ we have $K_\tau\to K_0$, by Lemma \ref{L7.5.2}. For sufficiently large $t>0$, this means the convergence $K_\tau\cap C_t\to K_0\cap C_t$ of convex bodies, and this implies the weak convergence
$$ S_{n-1}(K_\tau\cap C_t,\cdot)  \xrightarrow{w} S_{n-1}(K_0\cap C_t,\cdot),$$
equivalently
$$ \int_{\Sn} F\,\D S_{n-1}(K_\tau\cap C_t,\cdot) \to \int_{\Sn} F\,\D S_{n-1}(K_0\cap C_t,\cdot)$$
for every continuous function $F:\Sn\to\R$. Given a continuous function $h:\omega\to\R$, there is (by Tietze's extension theorem) a continuous function $F:\Sn\to\R$ with $F=h$ on $\omega$ and $F=0$ on $\Sn\setminus\Omega_C$. In view of (\ref{98}), it follows that
$$ \int_\omega h\,\D S_{n-1}(K_\tau,\cdot) \to \int_\omega h\,\D S_{n-1}(K_0,\cdot)$$
as $\tau\to 0$. Now (\ref{13b}) and (\ref{101}) yield
\begin{equation}\label{104}
\frac{1}{n} \int_\omega g(u)\,S_{n-1}(K_0,\D u) \ge \limsup_{\tau\downarrow 0} \frac{V_n(\tau)-V_{n-1}(\tau)}{\tau}.
\end{equation}

The Minkowski-type inequality (\ref{8}) gives
\begin{eqnarray*}
[V_1(\tau)-V_0]\sum_{k=0}^{n-1} [V_1(\tau)/V_0]^k &=& [V_1(\tau)^n-V_0^n]V_0^{1-n}\\
&\le & [V_0^{n-1}V_n(\tau)-V_0^n]V_0^{1-n}= V_n(\tau)-V_0.
\end{eqnarray*}
For $\tau\to 0$ we have $K_\tau\to K_0$ and hence, as follows from (\ref{5.0}), $V_1(\tau)\to V_0$. Therefore, we deduce that
\begin{equation}\label{105}
n\liminf_{\tau\downarrow 0} \frac{V_1(\tau)-V_0}{\tau} \le \liminf_{\tau\downarrow 0} \frac{V_n(\tau)-V_0}{\tau}.
\end{equation}
Replacing the pair $(V_1(\tau),V_0)$ by $(V_{n-1}(\tau),V_n(\tau))$, we can argue in a similar way. Here we have to observe that $V_{n-1}(\tau)/V_n(\tau)\to 1$ as $\tau\downarrow 0$, by the weak continuity of $S_{n-1}$ and the uniform convergence $h(K_\tau,\cdot)\to h(K_0,\cdot)$ (on $\omega$). We obtain
\begin{equation}\label{106}
n \limsup_{\tau\downarrow 0} \frac{V_n(\tau)-V_{n-1}(\tau)}{\tau} \ge \limsup_{\tau\downarrow 0}\frac{V_n(\tau)-V_0}{\tau}.
\end{equation}

Applying successively (\ref{104}), (\ref{106}), $\limsup \ge \liminf$, (\ref{105}), (\ref{103}), we conclude that in all these inequalities the equality sign is valid, hence
$$ \lim_{\tau\downarrow 0} \frac{V_n(\tau)-V_0}{\tau} =\int_\omega g(u)\,S_{n-1}(K_0,\D u).$$
This yields the main assertion of the lemma. The corresponding assertion for left-sided derivatives follows upon replacing $G(\tau,u)$ by $G(-\tau,u)$, and both results together give the corresponding result for limits.
\end{proof}

\section{Proof of Theorem 3}\label{sec9}

The proof of Theorem \ref{T3} requires a preparation, to ensure the existence of a maximum that we need. As in Section \ref{sec8}, $\omega\subset\Omega_C$ is a nonempty, compact set and $\K(C,\omega)$ is the set of closed convex sets $K$ that are $C$-determined by $\omega$.

\begin{lemma}\label{L7}
To every bounded set $B\subset C$ there is a number $t>0$ with the following property. If
$$ H(u,\tau)\cap B\not=\emptyset \quad\mbox{with } u\in\omega,$$
then $H(u,\tau)\cap C\subset C_t$.

There is a constant $t>0$ with the following property. If 
$$ K\in\K(C,\omega)\quad\mbox{and} \quad V_n(C\setminus K)=1,$$ then
$$ C\cap H_t\subset K. $$
\end{lemma}

\begin{proof}
Let $B\subset C$  be a bounded set. We choose $s$ with $B\subset C_s$. Let $u\in\omega$ be given. Let $H(u,\sigma)$ (with $\sigma<0$) be the supporting hyperplane of $C_s$ such that $C_s\subset H^+(u,\sigma)$. There is a point $z\in H(u,\sigma)\cap C_s$. Let $x$ be a point of maximal norm in $H(u,\sigma)\cap C$. Then $x\in{\rm bd}\,C$. Through $x$ there is a supporting hyperplane $H(v,0)$ of $C$. Its normal vector $v$ is in the relative boundary of $\Omega_C$, hence the angle $\psi$ between $u$ and $v$ satisfies $\sin\psi\ge a_0$, with $a_0$ as in (\ref{1}). The segment $[o,x]$ meets $H_s$ in a point $y$. Let the triangle with vertices $x,y,z$ have angle $\alpha$ at $x$ and angle $\gamma$ at $z$. Then $\alpha$ is the angle between a line in $H(u,\sigma)$ passing through $x$ and a line in $H(v,0)$ passing through $x$, hence $\alpha$ is not smaller than the angle between these hyperplanes, which is the angle $\psi$ between $u$ and $v$. It follows that
$$ \|x-y\| = \frac{\sin\gamma}{\sin\alpha} \|y-z\| \le \frac{{\rm diam}\,C_s}{a_0}.$$
Hence, we can choose a ball with center $o$, independent of $u$, that contains  $H(u,\sigma)\cap C$. Then we can choose $t>0$ such that $H(u,\sigma)\cap C\subset C_t$. If $H(u,\tau)$ meets $B$, then $\tau\ge \sigma$ and hence $H(u,\tau)\cap C\subset C_t$. This proves the first part.

For the second part, we choose a number $\zeta>0$ with $V_n(C_\zeta)>1$. By the first part, there is a number $t>0$ such that every hyperplane $H(u,\tau)$ with $H(u,\tau)\cap C_\zeta\not=\emptyset$ and $u\in\omega$ satisfies $H(u,\tau)\cap C\subset C_t$.

Let $K\in\K(C,\omega)$ be a closed convex set with $V_n(C\setminus K)=1$. Then there is some point $x\in K\cap C_\zeta$. Hence, every supporting hyperplane $H(u,\tau)$ of $K$ with outer normal vector $u\in\omega$ satisfies $H(u,\tau)\cap C_\zeta\not=\emptyset$ and, therefore, $H(u,\tau)\cap C\subset C_t$. This shows that $\tau(K,\omega)\subset C_t$, which implies $C\cap H_t\subset K$.
\end{proof}

To prove now Theorem \ref{T3}, let $\varphi$ be a nonzero, finite Borel measure on $\Omega_C$ that is concentrated on $\omega$ (that is, $\varphi(\Omega_C\setminus \omega)=0$).

Let $\C(\omega)$ denote the set of positive, continuous functions on $\omega$, equipped with the topology induced by the maximum norm. For $f\in \C(\omega)$, let $K_f$ be the Wulff shape associated with $(C,\omega,f)$, and write $V(f)=V_n(C\setminus K_f)$. Define a function $\Phi: \C(\omega)\to (0,\infty)$ by
$$ \Phi(f):= V(f)^{-1/n} \int_\omega f\,\D\varphi.$$
The function $\Phi$ is continuous. We show first that it attains a maximum on the set ${\mathcal L}':= \{-h_L: L\in\K(C,\omega),\, V_n(C\setminus L)=1\}$.

Let $L\in\K(C,\omega)$ be such that $-h_L\in\cL'$. By Lemma \ref{L7}, there is a number $t>0$ such that $C\cap H_t\subset L$. This implies that
$$ \Phi(-h_L) \le \int_\omega -h(C\cap H_t,u)\,\varphi(\D u) =: c,$$
which is independent of $L$. It follows that $\sup\{\Phi(f): f\in{\mathcal L}'\}<\infty$.

Let $(K_i)_{i\in\N}$ be a sequence with $-h_{K_i}\in\cL'$ such that
$$ \lim_{i\to\infty} \Phi(-h_{K_i}) = \sup\{\Phi(f): f\in{\mathcal L}'\}.$$
For each $i$ we have $C\cap H_t\subset K_i$, hence $K_i\cap H^-_t\not=\emptyset$. By the Blaschke selection theorem, the bounded sequence $(K_i\cap H^-_t)_{i\in\N}$ of convex bodies has a subsequence converging to some convex body. Therefore, the sequence $(K_i)_{i\in\N}$ converges to a $C$-full set $K_0$. This set satisfies $V_n(C\setminus K_0)=1$ and hence $-h_{K_0}$ belongs to ${\mathcal L}'$, as follows from Lemma \ref{L3.2}. By continuity, the functional $\Phi$ attains its maximum on $\cL'$ at $-h_{K_0}$.

Since the functional $\Phi$ is homogeneous of degree zero, also its maximum on the larger set ${\mathcal L}:=\{-h_L: L\in\K(C,\omega)\}$ is attained at $-h_{K_0}$.

Let $f\in \C(\omega)$, and let $K_f$ be the Wulff shape associated with $(C,\omega,f)$. Then we have $-h_{K_f}(u)\ge f(u)$ for $u\in\omega$ and $V(-h_{K_f})=V(f)$. Therefore, $\Phi(f)\le \Phi(-h_{K_f})\le \Phi(-h_{K_0})$. Thus, also the maximum of $\Phi$ on the set $\C(\omega)$ is attained at $-h_{K_0}$.

Let $f\in \C(\omega)$. Then $-h_{K_0}+\tau f\in \C(\omega)$ for sufficiently small $|\tau|$, hence the function
\begin{equation}\label{4.1} 
\tau \mapsto \Phi(-h_{K_0}+\tau f) = V(-h_{K_0}+\tau f)^{-1/n}\left(\int_\omega -h_{K_0}\,\D\varphi+\tau \int_\omega f\,\D\varphi\right)
\end{equation}attains a maximum at $\tau=0$. Lemma \ref{L7.5.3} yields
$$ \frac{\D}{\D\tau} V(-h_{K_0}+\tau f)\Big|_{\tau=0} =\int_\omega f\,\D S_{n-1}({K_0},\cdot).$$
Therefore, and since $V(-h_{K_0})=1$, the derivative of the function (\ref{4.1}) at $\tau=0$ is given by
$$ -\int_\omega f\,\D S_{n-1}({K_0},\cdot) \cdot\frac{1}{n}\int_\omega -h_{K_0}\,\D\varphi + \int_\omega f\,\D\varphi,$$
and this is equal to zero. With
$$ \lambda:=\frac{1}{n} \int_\omega -h_{K_0}\,\D\varphi$$
we have
$\lambda>0$ and
$$\int_\omega f\,\D\varphi =\lambda \int_\omega f\,\D S_{n-1}({K_0},\cdot).$$
Since this holds for all functions $f\in \C(\omega)$, it holds for all continuous real functions $f$ on $\omega$. This yields $\varphi=\lambda S_{n-1}({K_0},\cdot) = S_{n-1}(\lambda^{\frac{1}{n-1}}{K_0},\cdot)$. Thus, $\varphi$ is the surface area  measure of the $C$-full set $\lambda^{\frac{1}{n-1}}{K_0}$.

\section{The cone-volume measure of a $C$-close set}\label{sec10}

In this section, we introduce the cone-volume measure of a $C$-close set $K$. Although we use the same terminology, this has to be distinguished from the cone-volume measure of a convex body containing the origin in the interior. Since $o\notin \,K$, there is little danger of ambiguity.

By $\B(\Omega_C)$ we denote the $\sigma$-algebra of Borel sets in $\Omega_C$.

Let $K$ be a $C$-close set and let $\omega\in \B(\Omega_C)$. If $x\in\tau(K,\omega)$, then the half-open segment $[o,x)$ belongs to $C\setminus K$. We define
$$ M(K,\omega)= \bigcup_{x\in \tau(K,\omega)}[o,x).$$
The set $M(K,\omega)$ is Lebesgue measurable. This can be shown by using arguments analogous to those in the proof of \cite[Lemma 2.2.13]{Sch14}): the system of all subsets $\omega\subset \Omega_C$ for which $M(K,\omega)$ is Lebesgue measurable is a $\sigma$-algebra containing the closed sets, and hence all Borel subsets of $\Omega_C$. If $\omega_1\cap \omega_2=\emptyset$, then $M(K,\omega_1)\cap M(K,\omega_2)$ has measure zero, from which one can deduce that 
\begin{equation}\label{5.1}  
V_K(\omega):= \Ha^n(M(K,\omega)),\quad \omega\in\B(\Omega_C),
\end{equation}
defines a measure on $\Omega_C$. It is finite, by the definition of a $C$-close set. We call it the {\em cone-volume measure} of $K$.

\begin{lemma}\label{L3}
The cone-volume measure of the $C$-close set $K$ can be represented by
\begin{equation}\label{5.2} 
V_K(\omega) = \frac{1}{n} \int_\omega -h(K,u)\, S_{n-1}(K,\D u)
\end{equation}
for $\omega\in\B(\Omega_C)$.
\end{lemma}

We mention that in the theory of convex bodies the cone-volume measure of a convex body containing the origin is usually defined by an integral representation corresponding to (\ref{5.2}), whereas an interpretation corresponding to (\ref{5.1}), which justifies the name, is mentioned only for polytopes. An exception is \cite[Lemma 9.2.4]{Sch14}, which we follow here in some respects.

\begin{proof}
For any $C$-close set $K$, we define
$$ \psi_K(\omega) := \frac{1}{n} \int_\omega - h(K,u)\,S_{n-1}(K,\D u), \quad\omega\in\B(\Omega_C). $$
Then $\psi_K$ is a measure on $\Omega_C$.

First let $K$ be a $C$-full set, so that $C \setminus K$ is bounded. We choose $t>0$ with $C \setminus K\subset C_t$, and then we choose a dense sequence $(u_i)_{i\in\N}$ in $\Omega_C$ and define
\begin{equation}\label{20a} 
K_j:= C\cap \bigcap_{i=1}^j H^-(K,u_i)
\end{equation}
for $j\in\N$, where $H^-(K,u_i)$ is the supporting halfspace of the closed convex set $K$ with outer normal vector $u_i$. Then $K_j$ is a $C$-full set. From the denseness of the sequence $(u_i)_{i\in\N}$ it follows that 
$$ K_j\to K$$ 
as $j\to\infty$. This implies
$$ \lim_{j\to\infty} V_n(C\setminus K_j) = V_n(C\setminus K)$$
(by the continuity of the volume of convex bodies, applied to $K_j\cap C_t$). In view of (\ref{20a}), it is an elementary matter to show that
\begin{equation}\label{3.20} 
V_{K_j}=\psi_{K_j}
\end{equation}
(since $K_j$ is the intersection of $C$ with a polyhedron).

We show the weak convergence
\begin{equation}\label{3.22}
V_{K_j} \xrightarrow{w} V_K,\quad j\to\infty.
\end{equation}
(See, e.g., Ash \cite[Sect. 4.5]{Ash72}, for equivalent definitions of weak convergence of finite Borel measures on metric spaces.) For this, we define the radial function $\rho(K,\cdot)$ of $K$ on $C\cap\Sn$ by
$$ \rho(K,u) := \sup\{r\ge 0: ru \in C\setminus K\}, \quad u\in C\cap \Sn.$$
Let $\nu(x):=x/\|x\|$ for $x\in\R^n\setminus\{o\}$. By using polar coordinates, we obtain
\begin{eqnarray*}
V_K(\omega) = \Ha^n\left(\bigcup_{x\in \tau(K,\omega)} [o,x)\right) &=& \int_{\nu(\tau(K,\omega))} \int_0^{\rho(K,u)} s^{n-1}\,\D s\,\Ha^{n-1}(\D u)\\
&=& \frac{1}{n} \int_{\nu(\tau(K,\omega))} \rho(K,u)^n\,\Ha^{n-1}(\D u).
\end{eqnarray*}
Similarly,
$$ V_{K_j}(\omega) = \frac{1}{n} \int_{\nu(\tau(K_j,\omega))} \rho(K_j,u)^n\,\Ha^{n-1}(\D u).$$
For $\Ha^{n-1}$-almost all $u\in C\cap\Sn$, the outer unit normal vector $n(K_j,u)$ of $K_j$ at the boundary point $\rho(K_j,u)u$ is uniquely determined for all $j\in\N_0$, as follows from \cite[Thm. 2.2.5]{Sch14}, applied to the countably many convex bodies $K_j\cap C_t$. Since $K_j\to K$, for almost all $u\in C\cap\Sn$ we have $n(K_j,u)\to n(K,u)$ for $j\to\infty$. For an open set $\omega\subset\Omega_C$, this implies that for almost all $u\in C\cap\Sn$, the inequality
$$ {\mathbbm 1}_{\nu(\tau(K,\omega))}(u) \le \liminf_{j\to\infty} {\mathbbm 1}_{\nu(\tau(K_j,\omega))}(u)$$
holds. Fatou's lemma and the continuous dependence of $\rho(K,\cdot)$ on $K\cap C_t$ give
$$ V_K(\omega) \le \liminf_{j\to\infty} V_{K_j}(\omega).$$
Since
\begin{eqnarray*}
V_K(\Omega_C) &=& V_n(C\setminus K) = V_n(C_t)-V_n(K\cap C_t)\\
& =& \lim_{j\to\infty}(V_n(C_t)-V_n(K_j\cap C_t))= \lim_{j\to\infty} V_{K_j}(\Omega_C),
\end{eqnarray*}
this completes the proof of the weak convergence (\ref{3.22}).

Next, we show the weak convergence
\begin{equation}\label{3.21} 
\psi_{K_j} \xrightarrow{w} \psi_K,\quad j\to\infty,
\end{equation}
of the finite measures $\psi_{K_j}$. For $\omega\in\B(\Omega_C)$ we have
\begin{eqnarray*}
\psi_K(\omega) &=& \frac{1}{n} \int_\omega -h(K,u)\,S_{n-1}(K,\D u)\\
&=& \frac{1}{n} \int_\omega |h(K\cap C_t,u)|\,S_{n-1}(K\cap C_t,\D u)
\end{eqnarray*}
and similarly
$$ \psi_{K_j} (\omega) = \frac{1}{n} \int_\omega |h(K_j\cap C_t,u)|\,S_{n-1}(K_j\cap C_t,\D u).$$

For $j\in\N_0$, let $\eta_j$ be the measure defined by
\begin{eqnarray*} 
\eta_j(\omega) &:=& \frac{1}{n} \int_\omega|h(K\cap C_t,u)|\,S_{n-1}(K_j\cap C_t,\D u),\quad \omega\in\B(\Omega_C),\,j\in\N,\\
\eta_0(\omega) &:=& \frac{1}{n} \int_\omega|h(K\cap C_t,u)|\,S_{n-1}(K\cap C_t,\D u),\quad \omega\in\B(\Omega_C).
\end{eqnarray*}
Since $h(K\cap C_t,\cdot)$ is continuous and the area measure $S_{n-1}$ is weakly continuous, we have $\eta_j \xrightarrow{w} \eta_0$ as $j\to\infty$. By \cite[Lemma 1.8.14]{Sch14}, the sequence $(h(K_j\cap C_t,\cdot))_{j\in\N}$  converges uniformly on $\Sn$ to $h(K\cap C_t,\cdot)$. Hence, for each $\varepsilon>0$ we have $|h(K\cap C_t,\cdot)|\le |h(K_j\cap C_t,\cdot)|+\varepsilon$ for all $u\in\Sn$ and hence $\eta_j(\omega)\le \psi_{K_j}(\omega)+c\varepsilon$, if $j$ is sufficiently large; here $c$ is a constant independent of $j$. Since this holds for all $\varepsilon>0$ and since $\eta_j \xrightarrow{w} \eta_0$, we deduce that for each open set $\omega\subset\Omega_C$ we get
$$ \psi_K(\omega)= \eta_0(\omega) \le \liminf_{j\to\infty} \eta_j(\omega) \le \liminf_{j\to\infty} \psi_{K_j}(\omega).$$
Using Lemma \ref{L1},
\begin{eqnarray*} 
\psi_K(\Omega_C) &=& V_n(C\setminus K) = V_n(C_t)-V_n(K\cap C_t)\\
& =& \lim_{j\to\infty}(V_n(C_t)-V_n(K_j\cap C_t))= \lim_{j\to\infty} \psi_{K_j}(\Omega_C).
\end{eqnarray*}
This completes the proof of the weak convergence (\ref{3.21}).

From (\ref{3.20}), (\ref{3.22}), (\ref{3.21}) we now conclude that that $V_K=\psi_K $. This is the assertion of the lemma for a $C$-full set $K$.

Now let $K$ be a $C$-close set for which $C\setminus K$ is unbounded. Consider an open set $\omega\subset\Omega_C$ with ${\rm clos}\,\omega\subset\Omega_C$. Then the set $\tau(K,\omega)$ is bounded. The set
$$ M:= C\cap\bigcap_{u\in{\rm clos}\,\omega} H^-(K,u)$$
is $C$-full and satifies $\tau(K,\omega)=\tau(M,\omega)$ and hence $V_K(\omega) = V_M(\omega)$, moreover, $\psi_K(\omega)= \psi_M(\omega)$. Since $C\setminus M$ is bounded, we have $V_M(\omega)=\psi_M(\omega)$. Therefore, $V_K(\omega)= \psi_K(\omega)$. Since $V_K$ and $\psi_K$ are both measures on $\Omega_C$, the equality $V_K(\omega)=\psi_K(\omega)$ extends to arbitrary Borel subsets $\omega\in\B(\Omega_C)$.
\end{proof}

\section{Proof of Theorem 4}\label{sec11}

As before, let $\omega\subset\Omega_C$ be a nonempty, closed set, and let $\C(\omega)$ denote the space of positive, continuous functions on $\omega$. Let $\varphi$ be a nonzero, finite Borel measure on $\Omega_C$ that is concentrated on $\omega$. To prove Theorem \ref{T4}, we follow the procedure in the proof of Theorem \ref{T3}, modified in the way the logarithmic Minkowski problem was treated by B\"or\"oczky, Lutwak, Yang and Zhang in \cite{BLYZ13}. 

Without loss of generality, we assume that $\varphi(\omega)=1$. Define a function $\Phi: \C(\omega)\to (0,\infty)$ by
$$ \Phi(f) := V(f)^{-1/n}\, \exp\int_\omega \log f\,\D\varphi, \quad f\in \C(\omega).$$

The following assertions are verified precisely as in the proof of Theorem \ref{T3}. The functional $\Phi$ is continuous and attains a maximum on the set $\cL':= \{-h_L: L\in\K(C,\omega),\, V_n(C\setminus L)=1)\}$, say at $-h_{K_0}$. Since $\Phi$ is homogeneous of degree zero (here it is used that $\varphi(\omega)=1$) this is also the maximum of $\Phi$ on the set $\cL:= \{-h_L:L\in\K(C,\omega)\}$. Let $f\in \C(\omega)$, and let $K_f$ be the Wulff shape associated with $(C,\omega,f)$. Then $-h_{K_f} \ge f$ and $V(f)=V(-h_{K_f})$, hence $\Phi(f)\le \Phi(-h_{K_f}) \le
\Phi(-h_{K_0})$, since $-h_{K_f}\in\cL$. Thus, $\Phi$ attains its maximum on $\C(\omega)$ at $-h_{K_0}$.

Now let $f\in \C(\omega)$ and define
$$ G(\tau,\cdot):= -h_{K_0} e^{\tau  f}.$$
Then $G(\tau,\cdot)\in \C(\omega)$, hence the function
\begin{equation}\label{6.1a} 
\tau \mapsto \Phi(G(\tau,\cdot)) = V(G(\tau,\cdot))^{-1/n}\,\exp\int_\omega \log G(\tau,\cdot)\,\D\varphi
\end{equation}
attains its maximum at $\tau=0$. Since
$$ \frac{G(\tau,\cdot)-G(0,\cdot)}{\tau} \to -fh_{K_0} \quad \mbox{uniformly on $\omega$}$$
as $\tau\to 0$, we can conclude from Lemma \ref{L7.5.3} that
$$ \frac{\D}{\D \tau} V(G(\tau,\cdot))\Big|_{\tau=0} = \int_\omega -fh_{K_0}\,\D S_{n-1}(K_0,\cdot).$$
Therefore, the derivative of the function (\ref{6.1a}) at $\tau=0$ is given by
$$ \left[-\frac{1}{n} \int_\omega -fh_{K_0}\,\D S_{n-1}(K_0,\cdot) +\int_\omega f\,\D\varphi\right] \exp \int_\omega \log(-h_{K_0})\,\D\varphi$$
(note that $V(G(0,\cdot)) = V(-h_{K_0}) = V_n(C\setminus K_0)=1$). Since this is equal  to zero, we obtain, in view of Lemma \ref{L3},
$$ \int_\omega f \,\D V_{K_0} = \int_\omega f\,\D\varphi.$$
Since this holds for all $f\in \C(\omega)$, we conclude that $V_{K_0}=\varphi$.

\section{Proof of Theorem 5}\label{sec12}

Let $\varphi$ be a non-zero, finite Borel measure on $\Omega_C$. We choose a sequence $(\omega_j)_{j\in\N}$ of open sets in $\Omega_C$ such that 
$$ \varphi(\omega_1)>0,\quad {\rm clos}\,\omega_j\subset \omega_{j+1},\quad\bigcup_{j\in\N} \omega_j=\Omega_C. $$
For $j\in\N$, we define the measure $\varphi_j$ by
$$ \varphi_j(\eta):= \varphi(\eta\cap \omega_j) \quad\mbox{for }\eta\in\B(\Omega_C).$$
Then $\varphi_j$ is a nonzero, finite Borel measure that is concentrated on ${\rm clos}\,\omega_j$. By Theorem \ref{T4}, there exists a $C$-full set $K_j$ with
$$ \varphi_j=V_{K_j}.$$
We choose $t>0$ with $V_n(C_t)>\varphi(\Omega_C)$. If $K_j\cap C_t=\emptyset$ for some $j$, then 
$$ \varphi(\omega_j) = V_{K_j}(\omega_j) =V_{K_j}(\Omega_C) = V_n(C\setminus K_j) \ge V_n(C_t)>\varphi(\Omega_C) \ge \varphi(\omega_j),$$
a contradiction. Therefore, $K_j\cap C_t\not=\emptyset$ for all $j$. We choose an increasing sequence $(t_k)_{k\in\N}$ with $t_1\ge t$ and $t_k\uparrow \infty$ as $k\to\infty$. 

Let $j\in\N$. Since $(K_j\cap C_{t_1})_{j\in\N}$ is a bounded sequence of nonempty convex bodies, it has a convergent subsequence. Thus, there are a subsequence $(j_i^{(1)})_{i\in\N}$ of $(j)_{j\in\N}$ and a convex body $M_1$ such that
$$ K_{j_i^{(1)}}\cap C_{t_1}\to M_1 \quad\mbox{as }i\to\infty.$$
Similarly, there are a subsequence $(j_i^{(2)})_{i\in\N}$ of $(j_i^{(1)})_{i\in\N}$ and a convex body $M_2$ such that
$$ K_{j_i^{(2)}}\cap C_{t_2}\to M_2 \quad\mbox{as }i\to\infty.$$
By induction, we obtain, for each $k\in\N$, a subsequence $(j_i^{(k)})_{i\in\N}$ of $(j_i^{(k-1)})_{i\in\N}$ and a convex body $M_k$ such that
$$ K_{j_i^{(k)}}\cap C_{t_k}\to M_k \quad\mbox{as }i\to\infty.$$
Now we take the diagonal sequence $(j_i)_{i\in \N}= (j_i^{(i)})_{i\in\N}$. Then
$$ K_{j_i} \cap C_{t_k}\to M_k\quad \mbox{as }i\to\infty,$$
for each $k\in\N$.

We change the notation and write $K_i$ for $K_{j_i}$ and $\omega_i$ for $\omega_{j_i}$, Then $\omega_i\uparrow \Omega_C$ ($i\to\infty$) and
$$ K_i\cap C_{t_k} \to M_k\;(i\to\infty)\quad\mbox{for }k\in\N.$$

Let $1\le m<k$. Then, as $i\to\infty$,
$$ K_i\cap C_{t_m}\to M_m, \qquad K_i\cap C_{t_k}\to M_k,$$ 
and the latter implies $ K_i\cap C_{t_m}\to M_k\cap C_{t_m}$, thus
$$M_m= M_k\cap C_{t_m}.$$
If we define
$$ M:= \bigcup_{k\in\N} M_k,$$
then
$$ M\cap C_{t_k}= M_k\quad \mbox{for } k\in\N.$$
From this it follows that $M\subset C$ is a closed convex set.

Now let $j\in\N$ and let $\omega\subset\omega_j$ be some open set. The set $\tau(M,\omega)$ is bounded, hence there is some $k\in\N$ with $\tau(M,\omega)= \tau(M_k,\omega)$. From $K_i\cap C_{t_k} \to M_k$ and the weak continuity of the cone-volume measure, which was proved in Section \ref{sec10}, we have
$$ V_{M_k}(\omega) \le \liminf_{i\to\infty} V_{K_i\cap C_{t_k}}(\omega).$$
The definition of $K_i$ implies that $V_{K_i\cap C_{t_k}}(\omega)=\varphi(\omega)$, hence
$$ V_{M_k}(\omega)\le \varphi(\omega).$$
On the other hand, if $\beta\subset \omega_j$ is a closed set, then a similar argument gives
\begin{equation}\label{7.1} 
V_{M_k}(\beta)\ge \limsup_{i\to\infty} V_{K_i\cap C_{t_k}}(\beta) = \varphi(\beta).
\end{equation}
For a given closed set $\beta\subset\omega_j$, we can choose a sequence $(\beta_r)_{r\in\N}$ of open neighborhoods of $\beta$ with $\beta_r\subset\omega_{j+1}$ and $\beta_r\downarrow\beta$ as $r\to\infty$. As above, we then have $V_{M_k}(\beta_r) \le \varphi(\beta_r)$. Since $\beta_r\downarrow \beta$, this gives $V_{M_k}(\beta)\le\varphi(\beta)$. Together with (\ref{7.1}), this shows that $V_{M_k}(\beta)=\varphi(\beta)$. In particular, $V_{M_k}(\beta)\le \varphi(\Omega_C)<\infty$ for all sufficiently large $k$. It follows that $V_n(C\setminus M)\le\varphi(\Omega_C)<\infty$, thus $M$ is a $C$-close set. Therefore, its cone-volume measure $V_M$ is defined. For any closed set $\beta\subset\Omega_C$ satisfying $\beta\subset\omega_j$ for some $j\in\N$, we have shown that, for suitable $k$, $\tau(M,\omega)=\tau(M_k,\omega)$ and hence $V_M(\beta)= V_{M_k}(\beta)=\varphi(\beta).$ Since $V_M$ and $\varphi$ are Borel measures on $\Omega_C$, and $\omega_j\uparrow\Omega_C$, the equality $V_M(\beta)=\varphi(\beta)$ holds for every closed set $\beta\in\B(\Omega_C)$ and hence for every Borel set $\beta\in\B(\Omega_C)$. Thus, $M$ is a $C$-close set with cone-volume measure $\varphi$. This completes the proof of Theorem \ref{T5}.

\noindent Author's address:\\[1mm]
Rolf Schneider\\
Mathematisches Institut, Albert-Ludwigs-Universit\"at\\
Eckerstr. 1, D-79104 Freiburg i. Br.\\
Germany\\[1mm]
e-mail: rolf.schneider@math.uni-freiburg.de

\end{document}